\documentclass[letterpaper,12pt]{article}

\usepackage[english]{babel}
\usepackage[utf8]{inputenc}

\usepackage[margin=1.19 in, top=1.1 in, bottom= 1.2 in]{geometry}
\makeatletter
\g@addto@macro\normalsize{%
  \setlength\abovedisplayskip{7pt}
  \setlength\belowdisplayskip{7pt}
  \setlength\abovedisplayshortskip{7pt}
  \setlength\belowdisplayshortskip{7pt}
}
\makeatother

\usepackage{tocloft}

\usepackage{amsfonts}
\usepackage{mathrsfs}
\usepackage{bbm}
\usepackage{amssymb}
\usepackage{mathtools}

\usepackage{enumitem}
\setlist{nolistsep} 	
\usepackage{amsthm}

\usepackage{xcolor}
\definecolor{Color1}{rgb}{0.0, 0.42, 0.47}
\definecolor{Color2}{rgb}{0.78, 0.11, 0.0}

\usepackage{titlesec}
\titleformat{\section}
  {\large\center\bfseries}
  {\thesection.}{.7em}{}
\titlespacing*{\section}{0pt}{3.5ex plus 0ex minus 0ex}{1.5ex plus 0ex}
\titleformat{\subsection}
  {\center\bfseries}
  {\thesubsection.}{.7em}{}
\titlespacing*{\subsection}{0pt}{3.5ex plus 0ex minus 0ex}{1.5ex plus 0ex}
\titleformat{\subsubsection}
  {\center\bfseries}
  {\thesubsubsection.}{.7em}{}
\titlespacing*{\subsubsection}{0pt}{3.5ex plus 0ex minus 0ex}{1.5ex plus 0ex}

\addto\captionsenglish{}

\usepackage{titling}
\makeatletter
\renewenvironment{abstract}{
\begin{center}
{\bfseries \large\abstractname\vspace{\z@}}
\end{center}
\quotation
}
\makeatother

\usepackage[linktocpage=true]{hyperref}
\usepackage[capitalize]{cleveref}
\hypersetup{citecolor = Color1,colorlinks,
			linkcolor = black,
			urlcolor = Color2}

\newtheoremstyle{plain}{3mm}{3mm}{\slshape}{}{\bfseries}{.}{.5em}{}
\newtheoremstyle{definition}{2mm}{2mm}{}{}{\bfseries}{.}{.5em}{}
\theoremstyle{plain} 
	
\newtheorem{theorem}{Theorem}[section]
\newtheorem{problem}[theorem]{Problem}
\newtheorem{proposition}[theorem]{Proposition}
\newtheorem{question}[theorem]{Question}
\newtheorem{conjecture}[theorem]{Conjecture}
\newtheorem{lemma}[theorem]{Lemma}
\newtheorem{corollary}[theorem]{Corollary}
\theoremstyle{definition} 
\newtheorem{definition}[theorem]{Definition}
\newtheorem{remark}[theorem]{Remark}
\newtheorem{example}[theorem]{Example}

\theoremstyle{plain} 
\newcounter{MainTheoremCounter}

\theoremstyle{plain}
\newtheorem*{namedthm}{\namedthmname}
\newcounter{namedthm}
\makeatletter
	\newenvironment{named}[2]
	{\def\namedthmname{#1}
	\refstepcounter{namedthm}
	\namedthm[#2]\def\@currentlabel{#1}}
	{\endnamedthm}
\makeatother

\usepackage{chngcntr}

\numberwithin{equation}{section}

\allowdisplaybreaks

\newcommand{\Cesaro}{Ces\`{a}ro}
\newcommand{\Erdos}{Erd\H{o}s}
\newcommand{\Folner}{F\o{}lner}

\newcommand{\Sarkozy}{S\'{a}rk\"{o}zy}
\newcommand{\Katai}{K\'{a}tai}

\newcommand{\Turan}{Tur{\'a}n}

\newcommand{\Oh}{{\mathrm O}}
\newcommand{\oh}{{\mathrm o}}
\newcommand{\N}{\mathbb{N}}
\newcommand{\Z}{\mathbb{Z}}
\newcommand{\R}{\mathbb{R}}
\newcommand{\C}{\mathbb{C}}
\newcommand{\Q}{\mathbb{Q}}
\newcommand{\T}{\mathbb{T}}

\newcommand{\Hilb}{\mathcal{H}}

\newcommand{\define}[1]{{\itshape #1}}

\renewcommand{\epsilon}{\varepsilon}
\renewcommand{\leq}{\leqslant}
\renewcommand{\geq}{\geqslant}
\renewcommand{\setminus}{\backslash}
\renewcommand{\Re}{{\rm Re}}

\renewcommand{\P}{\mathbb{P}}
\renewcommand{\subset}{\subseteq}

\renewcommand{\d}{~\mathrm{d}}

\usepackage[normalem]{ulem}

\usepackage[makeroom]{cancel}

\newcommand{\A}{\mathcal{A}}

\newcommand{\1}{1}
\newcommand{\E}{\operatornamewithlimits{\mathbb{E}}}
\newcommand{\logE}{\operatornamewithlimits{\mathbb{E}^{\log{}}}}

\newcommand{\bN}{\mathbf{N}}
\newcommand{\bP}{\mathbf{P}}

\newcommand{\bPW}{\mathbf{P}_W}
\newcommand{\str}{\mathrm{str}}
\newcommand{\rnd}{\mathrm{rnd}}
\newcommand{\AMdens}{\mathrm{d}^{\nnhash}\hspace{-.01em}}
\usepackage{graphicx}
\newcommand{\nnhash}{
\ooalign{
\rotatebox[origin=c]{-45}{$\rule{0.042em}{.4em}\hspace{.06em}\rule{0.042em}{.4em}$}\cr
\rotatebox[origin=c]{45}{$\rule{0.042em}{.4em}\hspace{.06em}\rule{0.042em}{.4em}$}\cr
}
}

\author{By~~{\scshape Florian~K.~Richter}}
\date{\small \today}
\title{\bfseries Sums and products in sets of positive density}

\begin{document}

\maketitle
\begin{abstract}
We develop an analytic approach that draws on tools from Fourier analysis and ergodic theory to study Ramsey-type problems involving sums and products in the integers. Suppose $Q$ denotes a polynomial with integer coefficients. We establish two main results. First, we show that if $Q(1) = 0$, then any set of natural numbers with positive upper logarithmic density contains a pair of the form $\{x + Q(y), xy\}$ for some $x, y \in \mathbb{N} \setminus \{1\}$. 
Second, we prove that if $Q(0) = 0$, then any set of natural numbers with positive density relative to a new multiplicative notion of density, which arises naturally in the context of such problems, contains $\{x + Q(y), xy\}$ for some $x, y \in \mathbb{N}$.
\end{abstract}

\tableofcontents
\thispagestyle{empty}


\section{Introduction}


In \cite{Hindman79a}, Hindman posed the following conjecture (see also \cite[Question~3]{HLS03} and \cite[Question~11]{Bergelson96}). 

\begin{conjecture}[Hindman's conjecture]
\label{conj_1}
For any finite coloring of $\N=\{1,2,3,\ldots\}$ there exist infinitely many $x,y\in\N$ such that the set $\{x,y,x+y,xy\}$ is monochromatic.
\end{conjecture}

In fact, Hindman also formulated an extension of \cref{conj_1} in \cite{Hindman79a} (see also~\cite[Question 17.18]{HS12a},~{\cite[p.~84]{GRS90}} and~{\cite[Problem~6.5]{BR09}}).

\begin{conjecture}[Generalized Hindman's conjecture]
\label{conj_2}
For any $k\in\N$ and any finite coloring of $\N$ there exist infinitely many $x_1<\ldots<x_k\in\N$ such that all finite sums and all finite products formed using distinct elements from $\{x_1,\ldots,x_k\}$ are monochromatic.
\end{conjecture}

Note that \cref{conj_1} corresponds to the case $k=2$ of \cref{conj_2}.

Hindman's conjectures stand among the most significant touchstone problems in Ramsey theory, exposing the limits of our understanding of the complex interplay between additive and multiplicative structures in the integers.
A major step towards resolving \cref{conj_1} was taken in \cite{Moreira17}, where it was shown that any finite coloring of $\N$ admits infinitely many $x,y\in\N$ such that the set $\{x,x+y,xy\}$ is monochromatic.
The special case of \cref{conj_1} where the coloring consists of two colors was solved recently in \cite{Bowen25}. Despite this progress, \cref{conj_1} remains open. Moreover, Hindman's conjectures are surrounded by a wide range of other unsolved Ramsey-type problems concerning the joint behavior of sums and products in the integers, see for example~\cite[Section~7]{Sahasrabudhe18}, \cite[Section~6]{BM18}, \cite[Section~5]{KMRR25}, \cite[Section~6]{BS24}, \cite[Section~4]{Alweiss24}, \cite{Frantzikinakis24}, and \cite[Section~4]{Bowen25}.

Considerably more is known about questions of this nature when the setting of the integers is replaced by the setting of fields, such as the finite field with $p$ elements $\mathbb{F}_p$, or the rational numbers $\Q$.
Shkredov~\cite{Shkredov10} used Fourier analysis to study the pattern $\{x,x+y,xy\}$ in $\mathbb{F}_p$, showing that for any $\delta>0$ there is a cofinite set of primes $p$ such that any subset of $\mathbb{F}_p$ with relative density $\geq\delta$ contains $\{x,x+y,xy\}$. 
Later, Green and Sanders~\cite{GS16} proved that
for any $r\in\N$ there is a cofinite set of primes $p$ such that any $r$-coloring of $\mathbb{F}_p$ admits a monochromatic quadruple $\{x, y, x+y, xy\}$; this result can be viewed as the natural analogue of \cref{conj_1} in finite fields. 
In \cite{BM17}, Bergelson and Moreira used techniques from ergodic theory to study monochromatic sums and products in countably infinite fields, such as $\Q$. They proved that any set with positive density (with respect to a double \Folner{} sequence, see \cite[Definition~1.3]{BM17}) contains a set of the from $\{x+y,xy\}$. In \cite{Kousek26}, Kousek refined the ergodic approach and extended its scope.
Finally, the analogues of Conjectures~\ref{conj_1} and~\ref{conj_2} in $\mathbb{Q}$ were solved recently in \cite{BS24} and \cite{Alweiss23}, respectively.

This goes to show that Fourier-analytic and ergodic-theoretic techniques have been used with great success to study the joint behavior of sums and products in fields. Despite this success, the application of these tools has not yet been adapted to the setting of the integers, as the non-amenability of affine integer actions presents serious obstructions. The main purpose of this paper is to use ideas from multiplicative number theory to overcome these obstacles and establish an analytic framework for studying Ramsey-type sum-product problems in $\N$ based on methods from Fourier analysis and ergodic theory. 

Our work is motivated by two open problems informed by Hindman's conjectures. The first is a conjecture formulated several years ago by Moreira.
The \define{density} of a set of natural numbers $A\subset\N$ is defined as
\begin{equation}
\label{eqn_density}
d(A)=\lim_{N\to\infty}\frac{|A\cap \{1,\ldots,N\}|}{N}
\end{equation}
whenever this limit exists.

\begin{conjecture}[Moreira's Conjecture]
\label{conj_3}
Any subset of $\N$ with positive density contains $\{x+y-1,xy\}$ for some $x,y\in\N\setminus\{1\}$. 
\end{conjecture}

Observe that not every subset of $\N$ with positive density contains $\{x+y,xy\}$ due to simple divisibility constraints; for instance, the set of odd numbers has density $\frac{1}{2}$ but does not contain such a configuration. 
Moreira's conjecture elegantly circumvents this issue via a translation of the first component, which dispels all such local obstructions.

The second motivating problem for this paper concerns the general question of which subsets of the integers contain the pattern $\{x+y,xy\}$. Very little is currently known about this question, with some conjectures 
being proposed in \cite[Conjectures 6.4~and~6.5]{BM18}.
A natural class to consider are all sets with positive density that exhibit no modular biases. More precisely, we say a set $A\subset\N$ is \define{evenly distributed across all residue classes} if for any $a\in\N$ and $b\in\N\cup\{0\}$ the density $d(A\cap (a\N+b))$ exists and equals $a^{-1}d(A)$. 

\begin{problem}
\label{prop_4}
Prove that any subset of $\N$ with positive density that is evenly distributed across all residue classes contains $\{x+y,xy\}$ for some $x,y\in\N$. 
\end{problem}

While both \cref{conj_3} and \cref{prop_4} have been circulating the field for some time, neither has appeared in print.
However, extensions of both have recently been asked in \cite[Questions~5.4 and~5.2]{KMRR25}.

\subsection{Main results}

Our main theorems resolve \cref{conj_3} and \cref{prop_4}. Before formulating them, we first introduce some necessary notation.

Given a finite set $B\subset\N$ and a function $f\colon B\to\C$, the \define{\Cesaro{} average} of $f$ over $B$ and the \define{logarithmic average} of $f$ over $B$ are defined receptively as
\[
\E_{n\in B} f(n)=\frac{\sum_{n\in B} f(n)}{|B|}
\qquad\text{and}\qquad
\logE_{n\in B} f(n)=\frac{\sum_{n\in B} \frac{f(n)}{n}}{\sum_{n\in B}\frac{1}{n}}.
\]
Throughout this paper, we use $[N]$ to abbreviate the set $\{1,\dots,N\}$.
The \define{upper density} of a set $A\subset\N$ is defined as
\[
\overline{d}(A)=\limsup_{N\to\infty} \E_{n\in[N]} \1_A(n),
\]
and the \define{upper logarithmic density} of $A$ is
\[
\overline{\delta}(A)=\limsup_{N\to\infty} \logE_{n\in[N]} \1_A(n).
\]
Note that for all $A\subset\N$ we have $\overline{\delta}(A)\leq \overline{d}(A)$, and if the density $d(A)$ exists (as defined in \eqref{eqn_density}) then
\begin{equation}
\label{eqn_equality_of_densities}
    d(A)=\overline{d}(A)=\overline{\delta}(A).
\end{equation}

As was already observed in the 1930's in the works of Besicovitch \cite{Besicovitch35} and Davenport-\Erdos{} \cite{DE36}, sets with positive upper logarithmic density have far richer multiplicative structure than sets with positive upper density (see also \cite{Erdos35,DE51} and \cite[Section 2]{BBHS06}).
Our first result reaffirms this heuristic, showing that in Moreira's conjecture the positive density assumption can be weakened to positive upper logarithmic density. Furthermore, our result provides a natural polynomial extension of \cref{conj_3}.

\begin{theorem}
\label{thm_1}
Let $Q$ be a polynomial with integer coefficients satisfying $Q(1)=0$. Then 
any $A\subset\N$ with $\overline{\delta}(A)>0$ contains $\{x+Q(y),xy\}$ for some $x,y\in\N\setminus\{1\}$.
\end{theorem}

The case $Q(y)=0$ of \cref{thm_1} recovers a classical result of Davenport and \Erdos{} \cite{DE36}, which asserts that any set of positive upper logarithmic density contains $\{x,xy\}$ for some $x,y\in\N\setminus\{1\}$.
By taking $Q(y)=y-1$ and in light of \eqref{eqn_equality_of_densities}, we see that \cref{thm_1} implies that \cref{conj_3} is true.

Our proof of \cref{thm_1} reveals that sets with positive upper logarithmic density admit not only one but many configurations of the from $\{x+Q(y),xy\}$ whenever $Q$ is an integer polynomial with $Q(1)=0$. In fact, we show that for any $A\subset\N$ and any $\epsilon>0$ there exist many ``almost prime'' numbers $y$ such that the set $\{x\in\N: x+Q(y),xy\in A\}$ has upper logarithmic density at least $\overline{\delta}(A)^2-\epsilon$. Loosely speaking, by ``almost prime'' numbers we mean positive integers that have relatively few prime factors; for the precise statement, see \cref{thm_1_tech}.

It is natural to ask wether the condition $\overline{\delta}(A)>0$ in \cref{thm_1} can be relaxed to the weaker condition $\overline{d}(A)>0$. The following example shows that this is not possible.

\begin{example}
\label{ex_1}
The set $A=\bigcup_{n\geq 4} [2^{2^n},2^{2^n+\frac{1}{2}})$ satisfies $\overline{d}(A)>0$, but does not contain patterns of the from $\{x+y-1,xy\}$ for $x,y\in\N\setminus\{1\}$.    
\end{example}

Our second main theorem can be seen as a first step towards a density version of Hindman's conjecture (\cref{conj_1}).  
We introduce a new density on the integers, denoted by $\AMdens$. For $A\subset\N$, it is defined as
\begin{equation}
\label{eqn_def_affine_density}
\AMdens(A)=\sup_{(N_s)_{s\in\N}}\Big(\liminf_{a\to\infty} \sup_{m\in a\N} \Big(\lim_{s\to\infty} \logE_{n\in[N_s]} \1_A(mn)\Big)\Big),
\end{equation}
where the outer supremum is taken over all increasing sequences $(N_s)_{s\in\N}$ such that for every $m\in\N$ the inside limit $\lim_{s\to\infty} \E^{\log}_{n\in[N_s]} \1_A(mn)$ exists.

Observe that writing $\liminf_{a\to\infty} \sup_{m\in a\N}$ is equivalent to taking a limit superior with respect to the partial ordering induced by the relation of divisibility on the positive integers, making it a natural property to consider. This suggests to view $\AMdens(A)$ as a measurement of the relative largeness of $A$ within progressions $a\N$ for highly divisible numbers $a\in\N$. In particular, if a set $A\subset\N$ satisfies $\AMdens(A)>0$ then necessarily $\overline{\delta}(A\cap a\N)>0$ for all $a\in\N$. 

Note that $\AMdens(.)$ possesses all the essential properties of a density, as it satisfies the following conditions:
\begin{itemize}
    \item[] \textit{unit range:}~~~$\AMdens(\emptyset)=0$ and $\AMdens(\N)=1$.
    \item[] \textit{monotonicity:}~if $A\subset B$ 
 then $\AMdens(A)\leq \AMdens(B)$.
    \item[] \textit{subadditivity:}~~for all $A,B\subset \N$ one has $\AMdens(A\cup B)\leq \AMdens(A)+\AMdens(B)$.
\end{itemize}
Moreover, the density $\AMdens(.)$ is multiplicatively invariant, meaning that for any $m\in\N$ and $A\subset\N$ we have $\AMdens(A/m)=\AMdens(A)$, where $A/m=\{n\in\N: nm\in A\}$.
Therefore, $\AMdens(.)$ cannot be additively invariant, since the integers don't support a density notion that is both additively and multiplicatively invariant, which is a consequence of the fact that the affine semigroup of integers is not amenable. This is one of the most crucial ways in which the setting of the integers differs from the setting of fields, where doubly-invariant notions of density are plentiful.
Nonetheless, the density $\AMdens(.)$ exhibits some additively invariant structure. Specifically, $\AMdens(.)$ is absolutely continuous with respect to $\overline{\delta}$, i.e., for all $A\subset\N$ we have $\overline{\delta}(A)=0\implies \AMdens(A)=0$.  

The following theorem provides some evidence that sets whose $\AMdens$-density is positive are both additively and multiplicatively rich.

\begin{theorem}
\label{thm_2}
Let $Q$ be a polynomial with integer coefficients satisfying $Q(0)=0$. Then any $A\subset\N$ with $\AMdens(A)>0$ contains $\{x+Q(y),xy\}$ for some $x,y\in\N$.
\end{theorem}

\cref{thm_2} is a consequence of a more general theorem that we prove, which provides optimal correlation estimates for the pattern $\{x+Q(y),xy\}$, see \cref{thm_2_tech} below.

Note that if $A\subset \N$ has positive density and is evenly distributed across all residue classes, then $\AMdens(A)>0$. This means that choosing $Q(y)=y$ in \cref{thm_2} resolves \cref{prop_4}.
Also, the subadditivity property of $\AMdens(.)$ implies that if $\N$ is finitely colored, say $\N=C_1\cup\ldots\cup C_r$, then at least one of the color classes $C_i$ satisfies $\AMdens(C_i)>0$. Therefore, for any finite coloring of $\N$ and any integer polynomial $Q$ with $Q(0)=0$ one can find infinitely many $x,y\in\N$ such that $\{x+Q(y),xy\}$ is monochromatic, recovering a special case of Moreira's Theorem~\cite[Theorem~1.4]{Moreira17}. This provides the first ``analytic'' proof that any finite coloring of $\N$ admits a monochromatic pair $\{x+y,xy\}$.

It is natural to ask whether in \eqref{eqn_def_affine_density} the restriction to sequences $(N_s)_{s\in\N}$ for which the inside limit exists is necessary, especially since without this restriction one would obtain a more appealing and easier to comprehend notion of density. However, it turns out that this restriction cannot be omitted. \cref{ex_2} provides the construction of a set $A\subset\N$ that satisfies
\begin{equation}
\label{eqn_bad_density}
\liminf_{a\to\infty} \sup_{m\in a\N} \Big(\limsup_{N\to\infty} \logE_{n\in[N]} \1_A(mn)\Big)\geq\frac{1}{8},
\end{equation}
but for which there exist no $x,y\in\N$ such that $\{x+y,xy\}\subset A$.

\begin{example}
\label{ex_2}
Let $N_{k,\ell}=2^{5^{k^2+\ell}}$, define
\[
A_{k,\ell} = 2^{2\ell-1}(2\N+1)\cap \big[N_{k,\ell}^{1/2}, N_{k,\ell}\big],
\]
and take $A=\bigcup_{\ell\leq k} A_{k,\ell}$.
If $m=2^{2\ell-1} q$ for some $q\in2\N+1$ then 
a straightforward calculation reveals that
\[
\limsup_{N\to\infty} \logE_{n\in [N]}\1_{A}(m n)\geq \lim_{k\to\infty} \logE_{n\in [N_{k,\ell}]}\1_{A_{k,\ell}}(2^{2\ell-1} q n)=\frac{1}{4}.
\]
On the other hand, if $m=2^{2\ell} q$ for some $q\in2\N+1$ then
\[
\limsup_{N\to\infty} \logE_{n\in [N]}\1_{A}(m n)\geq \lim_{k\to\infty} \logE_{n\in [N_{k,\ell}]}\1_{A_{k,\ell}}(2^{2\ell} q n)=\frac{1}{8}.
\]
Therefore, the set $A$ satisfies \eqref{eqn_bad_density}.
Note that for any $(k,\ell),(k',\ell')$ with $\ell\leq k$ and $\ell'\leq k'$, if either $k<k'$ or $\ell<\ell'$ then we have 
$N_{k,\ell}^4< N_{k',\ell'}$. 
This means that 
\[
x+y\in \big[N_{k,\ell}^{1/2}, N_{k,\ell}\big],~xy\in \big[N_{k',\ell'}^{1/2}, N_{k',\ell'}\big]\implies (k,\ell)=(k',\ell').
\]
So if $A$ contains $\{x+y,xy\}$, then necessarily there exist $(k,\ell)$ such that $A_{k,\ell}$ contains $\{x+y,xy\}$.
But the set $2^{2\ell-1}(2\N+1)$ cannot contain the sum and the product of the same two numbers. Indeed, all numbers in $2^{2\ell-1}(2\N+1)$ have identical $2$-adic valuation, and this $2$-adic valuation is an odd number. However, if the sum and the product of two given numbers have the same $2$-adic valuation, then this $2$-adic valuation must be even. We conclude that $2^{2\ell-1}(2\N+1)$ does not contain $\{x+y,xy\}$. Consequently, $A_{k,\ell}$ does not contain $\{x+y,xy\}$.
\end{example}

\subsection{Statement of main technical results}
\label{sec_main_technical_theorems}

While Theorems~\ref{thm_1} and~\ref{thm_2} are framed as the main results of this paper, we actually derive them from two stronger theorems formulated in this subsection.
These are averaging versions of Theorems~\ref{thm_1} and~\ref{thm_2}, whose formulations involve considerably more technical terminology, making them harder to state. However, they reveal more about the density aspects of the underlying problem.
In particular, they provide optimal correlation estimates for the patterns $\{x+Q(y),xy\}$ with respect to the corresponding averages involved (see~\cref{rem_2} for more details).

On the integers $\Z=\{\ldots,-2,-1,0,1,2,\ldots\}$, we define the shift map $\tau^b\colon \Z\to\Z$ for $b\in\Z$ and the dilation map $\sigma_a\colon\Z\to\Z$ for $a\in\N$ as
\begin{align}
\label{eqn_def_dilation_translation}    
\hspace*{3cm}\tau^b(n)=n+b\quad&\text{and}\quad \sigma_a(n)=an.
\intertext{Since $\sigma_{a_1a_2}=\sigma_{a_1}\circ\sigma_{a_2}$ and $\tau^{b_1+b_2}=\tau^{b_1}\circ \tau^{b_2}$ for all $a_1,a_1\in\N$ and $b_1,b_2\in\Z$, the multiplicative semigroup $(\N,\cdot)$ and the additive semigroup $(\Z,+)$ naturally act on $\Z$ via the representations $a\mapsto \sigma_a$ and $b\mapsto \tau^b$ respectively. In light of the distributive law}
\tag{\text{Distributive Law}}
\sigma_a\circ\tau^b&\,=\,\tau^{ab}\circ\sigma_a,
\end{align}
the joint action induced by $a\mapsto \sigma_a$ and $b\mapsto \tau^b$ corresponds to an action of the affine semigroup, i.e., the semigroup of maps $\{n\mapsto an+b: a\in\N,~b\in\Z\}$ under the operation of composition.

We denote by $\ell^\infty(\Z)$ the Banach space of all bounded complex-valued functions on $\Z$, endowed with the supremum norm $\|f\|_\infty=\sup_{n\in\Z} |f(n)|$.
The action on $\Z$ by the affine semigroup introduced above naturally extends to an anti-action on $\ell^\infty(\Z)$ via
\begin{align}
\label{eqn_affine_action_on_ellinfty}    
f\mapsto f\circ \sigma_a\quad&\text{and}\quad  f\mapsto f\circ\tau^b.
\end{align}

When endowed with pointwise addition and pointwise multiplication, the Banach space  $(\ell^\infty(\Z),\|.\|_\infty)$ becomes a (unital) $C^*$-algebra.
A \define{$C^*$-subalgebra} of $\ell^\infty(\Z)$
is a subset $\mathcal{A}\subset \ell^\infty(\Z)$
 satisfying the following four properties:
\begin{itemize}
    \item[--] $\mathcal{A}$ is closed (with respect to $\|.\|_\infty$);
    \item[--] $\1_\Z \in\A$;
    \item[--] if $f,g\in\A$ and $\lambda,\eta\in\C$ then $\lambda f+\eta g\in\A$;
    \item[--] if $f\in \A$ then $\overline{f}\in\A$.
\end{itemize}
Note that the second property in this list specifies that $\mathcal{A}$ is in fact a \define{unital} $C^*$-subalgebra, as $\1_\Z$ is a multiplicative unit. Since we consider only unital $C^*$-subalgebras throughout this paper, we omit the word ``unital'' from the definition. We caution the reader, however, that henceforth all occurrences of ``$C^*$-subalgebra'' should be understood to mean ``unital $C^*$-subalgebra.''

We denote by $\P=\{2,3,5,7,11,\ldots\}$ the set of prime numbers.

\begin{definition}
\label{def_algebra_properties}
Let $\bN=([N_s])_{s\in\N}$, where $N_1<N_2<\ldots\in\N$, and
let $\bPW=(\P_W\cap[M_{t}])_{t\in\N}$, where $M_1<M_2<\ldots\in\N$, $W\in\N$, and define
\[
\P_W=\{p\in\P: p\equiv 1\bmod W\}.
\]
If $\A\subset\ell^\infty(\Z)$ is a $C^*$-subalgebra then we say that $\A$~$\ldots$
\begin{itemize}
\item $\ldots$~is \define{separable} if it contains a countable, dense subset;
\item $\ldots$~is \define{translation invariant} if for all $f\in \A$ and $b\in\Z$ we have $f\circ \tau^b \in\A$;
\item $\ldots$~is \define{dilation invariant} if for all $f\in \A$ and $a\in\N$ we have $ f\circ\sigma_a\in\A$;
\item $\ldots$~is \define{affinely invariant} if it is both translation and dilation invariant.
\item $\ldots$~\define{admits logarithmic averages along $\bN$} if for all $f\in\A$ the limit
\[
\lim_{s\to\infty} \logE_{n\in [N_s]}f(n)
\]
exists.
In this case, we simply write
\begin{equation}
\label{eqn_inner_log_average_notation}
    \logE_{n\in\bN}f(n)=\lim_{s\to\infty} \logE_{n\in [N_s]}f(n)
\end{equation}
for all $f\in\A$, to streamline notation.
\item $\ldots$~\define{admits iterated affine correlations along $\bPW$} if for all $f,g\in\A$, $k,a\in\N$, and all polynomials $Q$ with integer coefficients, the limits
\begin{equation}
\label{eqn_iterated_averages_1}
\lim_{t_k\to\infty}\logE_{p_k\in\P_W\cap[M_{t_k}]}\cdots \lim_{t_1\to\infty} \logE_{p_1\in\P_W\cap[M_{t_1}]} \Big(\lim_{s\to\infty}\logE_{n\in [N_s]} f(a n+Q(p_k\cdots p_1))g(p_k\cdots p_1 n)\Big)
\end{equation}
exist. In this case, instead of \eqref{eqn_iterated_averages_1} we simply write
\begin{equation*}
\logE_{p\in\bPW^{*k}}\logE_{n\in\bN} f(an+Q(p))g(p n),
\end{equation*}
again to streamline notation.
In particular, we use $\E^{\log}_{p\in\bPW^{*k}}$ to abbreviate the iterated average $\E^{\log}_{p_k\in\bPW}\cdots \E^{\log}_{p_1\in\bPW}$.
When $W=1$, we simplify the notation further and write $\bP$ instead of $\bP_1$ and  $\E^{\log}_{p\in\bP^{*k}}$ instead of $\E^{\log}_{p\in\bP_1^{*k}}$.
\end{itemize}
\end{definition}

Let us now state our two main technical results, from which Theorems~\ref{thm_1} and~\ref{thm_2} can be derived.

\begin{theorem}
\label{thm_1_tech}
Let $N_1<N_2<\ldots\in\N$ and $M_1<M_2<\ldots\in\N$. Suppose $\A\subset\ell^\infty(\Z)$ is a separable, affinely invariant $C^*$-subalgebra that admits logarithmic averages along $\bN=([N_s])_{s\in\N}$ and iterated affine correlations along $\bPW=(\P_W\cap [M_t])_{t\in\N}$ for all $W\in\N$.
Let $Q$ be a polynomial with integer coefficients satisfying $Q(1)=0$.
Then for any $\1_A\in\A$ and $\epsilon>0$ there exist $k,W\in\N$ such that
\begin{equation}
\label{eqn_1_tech}
\logE_{p\in\bPW^{*k}} \logE_{n\in \bN} \1_A(n+Q(p))\1_A(pn) \geq \Big(\logE_{n\in\bN} \1_A(n)\Big)^2-\epsilon.
\end{equation}
\end{theorem}

\begin{proof}[Proof that \cref{thm_1_tech} implies \cref{thm_1}]
Suppose $A\subset\N$ with $\overline{\delta}(A)>0$ is given. Let $N_1<N_2<\ldots\in\N$ be any sequence such that $\overline{\delta}(A)=\lim_{s\to\infty}\E^{\log}_{n\in[N_s]}\1_A(n)$.
Let $\A\subset\ell^\infty(\Z)$ denote the smallest affinely invariant $C^*$-subalgebra that contains $\1_A$. Since finite linear combinations of products of translations and dilations of $\1_A$ are dense in $\A$, we see that $\A$ is separable. Using a standard diagonalization argument and replacing $(N_s)_{s\in\N}$ with a subsequence of itself if necessary, we can assume without loss of generality that $\A$ admits logarithmic averages along $\bN=([N_s])_{s\in\N}$. Let $M_1<M_2<\ldots\in\N$ be any sequence such that $\A$ admits iterated affine correlations along $\bPW=(\P_W\cap [M_t])_{t\in\N}$ for all $W\in\N$; again such a sequence exists because $\A$ is separable.
If we now apply \cref{thm_1_tech} then we obtain
\[
\logE_{p\in\bPW^{*k}} \logE_{n\in \bN} \1_A(n+Q(p))\1_A(pn) \geq \big(\overline{\delta}(A)\big)^2-\epsilon.
\]
So as long as $\epsilon<(\overline{\delta}(A))^2$, we can find some $k,W\in\N$, $p_1,\ldots,p_k\in\P_W$, and $n\in\N$ such that 
\[
\1_A(n+Q(p_1\cdots p_k))\1_A(p_1\cdots p_kn)>0.
\]
Taking $x=n$ and $y=p_1\cdots p_k$ proves $\{x+Q(y),xy\}\subset A$ as desired.
\end{proof}

We say a sequence $a_1,a_2,\ldots\in\N$ is \define{divisible} if for all $m\in\N$ the relation $a_i\equiv 0\bmod m$ holds for all but finitely many $i\in\N$.

\begin{theorem}
\label{thm_2_tech}
Let $N_1<N_2<\ldots\in\N$ and $M_1<M_2<\ldots\in\N$. Suppose $\A\subset\ell^\infty(\Z)$ is a separable, affinely invariant $C^*$-subalgebra that admits logarithmic averages along $\bN=([N_s])_{s\in\N}$ and iterated affine correlations along $\bP=(\P\cap [M_t])_{t\in\N}$.
Let $Q$ be a polynomial with integer coefficients satisfying $Q(0)=0$.
Suppose $\1_A\in\A$, let $a_1,a_2,\ldots\in\N$ be a divisible sequence, and define
\[
\delta=\limsup_{j\to\infty}\Big(\logE_{n\in\bN} \1_A(a_j n)\Big).
\]
Then for every $\epsilon>0$ there exist $k\in\N$ and $\{u,vu\}\subset \{a_j:j\in\N\}$ such that
\begin{equation}
\label{eqn_2_tech}
\logE_{p\in\bP^{*k}}\logE_{n\in\bN} \1_A(un+Q(vp))\1_A(uvpn) \geq \delta^2-\epsilon.
\end{equation}
\end{theorem}

\begin{proof}[Proof that \cref{thm_2_tech} implies \cref{thm_2}]
If $\AMdens(A)>0$ then by the definition of the density $\AMdens$ there exists a sequence $N_1<N_2<\ldots\in\N$ for which
\[
\liminf_{a\to\infty} \sup_{m\in a\N} \Big(\underbrace{\lim_{s\to\infty} \logE_{n\in[N_s]}}_{\E^{\log}_{n\in\bN} } \1_A(mn)\Big)>0.
\]
It follows that there exists a divisible sequence $a_1,a_2,\ldots\in\N$ such that
\[
\limsup_{j\to\infty}\Big(\logE_{n\in\bN} \1_A(a_j n)\Big)>0.
\]

As in the preceding proof, let $\A\subset\ell^\infty(\Z)$ denote the smallest affinely invariant $C^*$-subalgebra that contains $\1_A$, assume without loss of generality that $\A$ admits logarithmic averages along $\bN=([N_s])_{s\in\N}$, and let 
$M_1<M_2<\ldots\in\N$ be an arbitrary increasing sequence such that $\A$ admits iterated affine correlations along $\bP=(\P\cap [M_t])_{t\in\N}$.  
Invoking \cref{thm_2_tech} with $\epsilon$ sufficiently small, we can find $k\in\N$ and $u,v\in\N$ such that
\[
\logE_{p\in\bP^{*k}}\logE_{n\in\bN} \1_A(un+Q(vp))\1_A(uvpn) >0.
\]
In particular, there exist $p_1,\ldots,p_k\in\P$, and $n\in\N$ such that 
\[
\1_A(un+Q(vp_1\cdots p_k))\1_A(uvp_1\cdots p_kn)>0.
\]
Taking $x=un$ and $y=vp_1\cdots p_k$ shows that $\{x+Q(y),xy\}\subset A$.
\end{proof}

For an outline of the main steps in the proofs of Theorems~\ref{thm_1_tech} and~\ref{thm_2_tech} see \cref{sec_outline_of_proof}.

\begin{remark}
\label{rem_2}
The lower bounds provided on the right hand sides of \eqref{eqn_1_tech} and \eqref{eqn_2_tech} are essentially optimal. To see this, one can, for example, consider $A$ to be the set of multiplicatively even numbers, i.e., all positive integers with an even number of prime factors (counted with multiplicity). It then follows form the main result in \cite{Tao16} that equality holds for both \eqref{eqn_1_tech} and \eqref{eqn_2_tech} when setting $\epsilon=0$,
\end{remark}

\paragraph{Acknowledgments.}We thank Joel Moreira for providing helpful comments on an earlier draft of this paper. The author was supported by the Swiss National Science Foundation grant TMSGI2-211214.

\section{Preliminaries}
\label{sec_preliminaries}

In this section we collect preliminaries and preparatory results from number theory, ergodic theory, and Fourier analysis which are needed for the proofs of Theorems~\ref{thm_1_tech} and~\ref{thm_2_tech}.

\subsection{Basic properties of logarithmic averages}

We use logarithmic averages in the formulations of Theorems~\ref{thm_1_tech} and~\ref{thm_2_tech} and, as we have seen in \cref{ex_1}, it is in general not possible to replace them by \Cesaro{} averages.
This naturally prompts the question: which properties of logarithmic averages are necessary for our argument? An answer is given by the following lemma, which identifies a property of logarithmic averages that does not hold for \Cesaro{} averages, and which will be used repeatedly in the proofs of our main results.

\begin{lemma}
\label{lem_dilation_log_averages}
For any $q,N\in\N$ and any $1$-bounded $f\colon\Z\to\C$ we have
\[
\logE_{n\in [N]} f(n) = \logE_{n\in [N]} q \1_{q\mid n} f\big(\tfrac{n}{q}\big) + \Oh\bigg(\frac{\log q}{\log N}\bigg).
\]
\end{lemma}

\begin{proof}
We have
\begin{align*}
\sum_{n\in [N]} \frac{f(n)}{n}
= \sum_{\substack{n\in [qN] \\ q\mid n}} \frac{ f\big(\tfrac{n}{q}\big)}{\tfrac{n}{q}}
= \sum_{n\in [qN]} \frac{q \1_{q\mid n} f\big(\tfrac{n}{q}\big)}{n} 
= \sum_{n\in [N]} \frac{q \1_{q\mid n} f\big(\tfrac{n}{q}\big)}{n} +\Oh(\log q).
\end{align*}
Hence,
\begin{align*}
\logE_{n\in [N]} f(n)&= \frac{1}{\big(\sum_{n\in [N]}\frac{1}{n}\big)} \Bigg(\sum_{n\in [N]} \frac{f(n)}{n} \Bigg)
\\
&= \frac{1}{\big(\sum_{n\in [N]}\frac{1}{n}\big)} \Bigg( \sum_{n\in [N]} \frac{q \1_{q\mid n} f\big(\tfrac{n}{q}\big)}{n}\Bigg) + \Oh\bigg(\frac{\log q}{\log N}\bigg)
\\
&=\logE_{n\in [N]} q \1_{q\mid n} f\big(\tfrac{n}{q}\big) + \Oh\bigg(\frac{\log q}{\log N}\bigg),
\end{align*}
as desired.
\end{proof}

\subsection{Preliminaries from multiplicative number theory}

A corollary of the \Turan{}-Kubilius inequality (cf.~\cite[Chapter~4]{Elliott79}) states that for any finite set of primes $P\subset\P\cap [N]$ one has  
\begin{equation}
\label{eqn_TK}
\sum_{n\in[N]} \left(\, \sum_{p \in P} \mathbf{1}_{p \mid n} - \sum_{p \in P} \frac{1}{p} \right)^2 \leq 3N\Bigg( \sum_{p \in P} \frac{1}{p} \Bigg),
\end{equation}
where $\mathbf{1}_{p \mid n}$ denotes the function that is $1$ if $p$ divides $n$ and $0$ otherwise.  
Dividing \eqref{eqn_TK} by $N(\sum_{p \in P} 1/p )^2$, we obtain the following equivalent version:
\begin{equation}
\label{eqn_TK_2}
\E_{n\in[N]} \Big( \logE_{p \in P} p\mathbf{1}_{p \mid n} - 1 \Big)^2 \leq 3\Bigg( \sum_{p \in P} \frac{1}{p} \Bigg)^{-1} .
\end{equation}

Note that in \eqref{eqn_TK_2}, it is not possible to replace the logarithmic average in the variable~$p$ with a \Cesaro{} average. 
Since our main results utilize logarithmic averages in all variables, we require a variant of \eqref{eqn_TK_2} that uses logarithmic averages in the variable~$n$ too.
One way to obtain such a variant would be to simply derive it from \eqref{eqn_TK_2} using partial summation. Another approach -- one we take -- is to give a short and self-contained proof.

\begin{proposition}[Logarithmically averaged \Turan{}-Kubilius inequality]
\label{prop_log_TK}
Let $N\in\N$ and $P\subset \P\cap [N]$.
Then
\[
\logE_{n\in [N]}\Big(\logE_{p\in P} \big(p\1_{p\mid n} -1\big)\Big)^2
\leq 9 \Bigg(\sum_{p\in P}\frac{1}{p} \Bigg)^{-1}.
\]
\end{proposition}

\begin{proof}
By expanding the square, we obtain
\begin{align*}
\logE_{n\in [N]}\Big(\logE_{p\in P} p\1_{p\mid n} -1\Big)^2
&=
\logE_{n\in [N]} \Big(\logE_{p\in P} p\1_{\mathrm{lcm}(p,q)\mid n}\Big)^2
-2 \logE_{n\in [N]} \Big(\logE_{p\in P} p\1_{p\mid n}\Big) +1
\\
&=
\logE_{p,q\in P} pq\Big(\logE_{n\in [N]} \1_{\mathrm{lcm}(p,q)\mid n}\Big)
-2 \logE_{p\in P} p\Big(\logE_{n\in [N]} \1_{p\mid n}\Big) +1.
\end{align*}
For the first term, we can estimate from above using
\begin{align*}
\logE_{n\in [N]} \1_{\mathrm{lcm}(p,q)\mid n}
&= \frac{1}{\big(\sum_{n\in[N]}\frac{1}{n}\big)} \sum_{n\in[N]} \frac{\1_{\mathrm{lcm}(p,q)\mid n}}{n} 
\leq \frac{1}{\mathrm{lcm}(p,q)}.
\end{align*}
For the second term, we can establish a bound from below by 
\begin{align*}
\logE_{n\in [N]} \1_{p\mid n}
&= \frac{1}{\big(\sum_{n\in[N]}\frac{1}{n}\big)} \sum_{n\in[N]} \frac{\1_{p\mid n}}{n}
\\
&= \frac{1}{\big(\sum_{n\in[N]}\frac{1}{n}\big)} \sum_{1\leq  n \leq N/p} \frac{1}{pn}
\\
&= \frac{1}{p} - \frac{1}{\big(\sum_{n\in[N]}\frac{1}{n}\big)} \sum_{N/p<  n \leq N} \frac{1}{pn}
\\
&\geq  \frac{1}{p} - \frac{\log (p)}{p\log(N)},
\end{align*}
where the last estimate is quickly derived from a standard approximation of the harmonic sum $\frac{1}{2M}-\frac{1}{8M^2}\leq (\sum_{n\in[M]}\frac{1}{n})-\log(M)-\gamma\leq \frac{1}{2M}$, where $\gamma$ is the Euler-Mascheroni constant.
It follows that
\begin{align*}
\logE_{n\in [N]}\Big(\logE_{p\in P} p\1_{p\mid n} -1\Big)^2
&=
\logE_{p,q\in P} \frac{pq}{\mathrm{lcm}(p,q)}
-1 + 2\logE_{p\in P} \frac{\log (p)}{\log(N)}.
\end{align*}
To finish the proof, note that
\[
\logE_{p,q\in P} \frac{pq}{\mathrm{lcm}(p,q)} \leq 1+ \left( \sum_{p \in P} \frac{1}{p} \right)^{-1},
\]
and that
\[
\logE_{p\in P} \frac{\log (p)}{\log(N)}=
\frac{\big(\sum_{p\in P}\frac{\log(p)}{p}\big)}{\big(\sum_{p\in P}\frac{1}{p}\big)\log(N)}.
\]
By Mertens' first theorem and since $P\subset \P\cap [N]$, we have
\[
\sum_{p\in P}\frac{\log(p)}{p}\leq \sum_{p\in \P\cap[N]}\frac{\log(p)}{p}\leq  \log(N)+ 2.
\]
Therefore, 
\[
\logE_{p\in P} \frac{\log (p)}{\log(N)}
\leq 4 \left( \sum_{p \in P} \frac{1}{p} \right)^{-1}.
\]
The claim follows by combining the above estimates.
\end{proof}

The way we employ \cref{prop_log_TK} in the subsequent proofs is via the following two corollaries.

\begin{corollary}
\label{cor_log_TK}
For $W\in\N$ and any bounded function $f\colon\Z\to\C$ we have
\[
\limsup_{M\to\infty}\limsup_{N\to\infty}\Big|\logE_{n\in[N]} f(n) - \logE_{p\in\P_W\cap [M]}\logE_{n\in[N]} f(pn)\Big|=0.
\]
\end{corollary}

\begin{proof}
Using \cref{lem_dilation_log_averages} and the Cauchy-Schwarz inequality, we obtain
\begin{align*}
\Big|\logE_{n\in[N]} f(n) - \logE_{p\in\P_W\cap [M]}\logE_{n\in[N]} f(pn)\Big|
&=
\Big|\logE_{n\in[N]} f(n) - \logE_{p\in\P_W\cap [M]}\logE_{n\in[N]} p\1_{p\mid n} f(n)\Big|
\\
&=
\Big|\logE_{n\in[N]} f(n) \Big( \logE_{p\in\P_W\cap [M]} \big(1- p\1_{p\mid n} \big)\Big)\Big|
\\
&\leq
\Big(\logE_{n\in[N]} |f(n)|^2 \Big)^{\frac{1}{2}}
\Big(\logE_{n\in[N]}\Big| \logE_{p\in\P_W\cap [M]} \big(1- p\1_{p\mid n} \big)\Big|^2\Big)^{\frac{1}{2}}.
\end{align*}
The claim now follows from \cref{prop_log_TK} together with the fact that
\[
\sum_{p\in\P_W}\frac{1}{p}=\infty,
\]
which follows from the prime number theorem in arithmetic progressions.
\end{proof}

\begin{corollary}
\label{cor_log_TK_2}
For $W\in\N$ and any bounded function $f\colon\Z\to\C$ we have for any $k\in\N$ that
\[
\limsup_{M_k\to\infty}\cdots \limsup_{M_1\to\infty}\limsup_{N\to\infty}\Big|\logE_{n\in[N]} f(n) - \logE_{p_k\in\P_W\cap [M_k]}\cdots \logE_{p_1\in\P_W\cap [M_1]}\logE_{n\in[N]} f(p_k\cdots p_1n)\Big|=0.
\]
\end{corollary}

\begin{proof}
This follows by iterating \cref{cor_log_TK}.
\end{proof}

\subsection{Preliminaries concerning exponential sums along primes}

Throughout, we use the abbreviation $e(x)=e^{2\pi i x}$ for $x\in\R$. 
It follows from the classical works of Vinogradov \cite{Vinogradov57,Vinogradov58} and Rhin \cite{Rhin73} that for any real polynomial $Q(x)=a_dx^d+\ldots+a_1x+a_0$ with the property that at least one of the coefficients $a_1,\ldots,a_d$ is an irrational number one has
\begin{equation}
\label{eqn_rhin}
\lim_{M\to\infty}\E_{p\in\P\cap[M]} e(Q(p)) =0.
\end{equation}
For the proof of our main results, we require a version of~\eqref{eqn_rhin} using logarithmic averages; we include a proof for completeness.

\begin{lemma}
\label{lem_log_vinogradov}
For any $a\in\N$, $b\in\Z$, any non-constant polynomial with rational coefficients $Q$, and any irrational real number $\alpha$ we have
\begin{equation}
\label{eqn_log_vinogradov}
\lim_{M\to\infty}\logE_{p\in\P\cap[M]} \1_{a\Z+b}(p)\, e( Q(p)\alpha) =0.
\end{equation}
\end{lemma}

\begin{proof}
We will make use of the basic identity
\[
\1_{a\Z+b}(n)=\frac{1}{a}\sum_{r=0}^{a-1} e\bigg(\frac{(b-n)r}{a}\bigg).
\]
Define $Q_r(x)=Q(x)\alpha-\frac{xr}{a}$.
We thus have
\begin{align*}
\sum_{p\in\P\cap[M]} \1_{a\Z+b}(p)\,\frac{e(Q(p)\alpha)}{p}
&=\frac{1}{a}\sum_{r=0}^{a-1} e\bigg(\frac{br}{a}\bigg) \Bigg(
\sum_{p\in\P\cap[M]}  \frac{e(Q_r(p))}{p} \Bigg).
\end{align*}
Let $\rho\in (1,2)$ and pick $J\in\N$ such that $M\in (\rho^{J},\rho^{J+1}]$.
By the prime number theorem, $|\P\cap (\rho^j,\rho^{j+1}]|\sim\frac{\rho^j}{j\log(\rho)}$.
Hence
\begin{align*}
\bigg|\sum_{p\in\P\cap[M]} & \frac{e(Q_r(p))}{p}\bigg|
\\
&\leq \sum_{j=1}^{J}\bigg| \sum_{p\in\P\cap (\rho^j,\rho^{j+1}]} \frac{e(Q_r(p))}{p}\bigg|+\Oh(1)
\\
&\leq \sum_{j=1}^{J}\bigg| \frac{1}{\rho^j} \sum_{p\in\P\cap (\rho^j,\rho^{j+1}]} e(Q_r(p))\bigg|+ \Oh\bigg(\sum_{j=1}^{J}\frac{\rho-1}{j\log(\rho)}\bigg)
\\
&= \sum_{j=1}^{J}\frac{1}{j\log(\rho)}\bigg|\frac{j\log(\rho)}{\rho^j} \sum_{p\in\P\cap (\rho^j,\rho^{j+1}]} e(Q_r(p))\bigg|+ \Oh\bigg(\sum_{j=1}^{J}\frac{\rho-1}{j\log(\rho)}\bigg).
\end{align*}
Note that by \eqref{eqn_rhin},
\begin{align*}
\lim_{j\to\infty} \frac{j\log(\rho)}{\rho^j} \sum_{p\in\P\cap (\rho^j,\rho^{j+1}]} e(Q_r(p))=\lim_{j\to\infty}
 \E_{p\in\P\cap (\rho^j,\rho^{j+1}]} e(Q_r(p))=
0,
\end{align*}
because $Q_r(x)$ is a polynomial with at least one irrational non-constant coefficient. Therefore, we are left with
\[
\sum_{j=1}^{J}\frac{1}{j\log(\rho)}\bigg|\frac{j\log(\rho)}{\rho^j} \sum_{p\in\P\cap (\rho^j,\rho^{j+1}]} e(Q_r(p))\bigg|=\oh_{J\to\infty}\bigg(\sum_{j=1}^{J}\frac{1}{j\log(\rho)}\bigg).
\]
Overall, this gives 
\begin{equation}
\label{eqn_log_vinogradov_2}
\bigg|\sum_{p\in\P\cap[M]} \1_{a\Z+b}(p)\,\frac{e(Q(p)\alpha)}{p}\bigg|=\Oh\bigg(\sum_{j=1}^{J}\frac{\rho-1}{j\log(\rho)}\bigg). 
\end{equation}
Using $M\in (\rho^{J},\rho^{J+1}]$ and the prime number theorem, we see that
\begin{equation}
\label{eqn_log_vinogradov_3}
\sum_{p\in\P\cap[M]} \frac{1}{p}\sim \sum_{j=1}^{J}\frac{1}{j\log(\rho)}.
\end{equation}
Combining \eqref{eqn_log_vinogradov_2} and \eqref{eqn_log_vinogradov_3}, we obtain
\[
\limsup_{M\to\infty}\Big|\logE_{p\in\P\cap[M]} \1_{a\Z+b}(p)\, e(Q(p)\alpha)\Big|=\Oh\big(\rho-1\big).
\]
Letting $\rho$ approach $1$ finishes the proof.
\end{proof}

Below, we offer a slight variant of \cref{lem_log_vinogradov}, which will be used in \cref{sec_proof_of_thm_rnd_component}. Recall that $\P_W=\{p\in \P: p\equiv 1\bmod W\}$. 

\begin{corollary}
\label{cor_log_vinogradov}
For any $a,W\in\N$, $b\in\Z$, any non-constant polynomial with rational coefficients $Q$, and any irrational real number $\alpha$ we have
\begin{equation}
\label{eqn_log_vinogradov_4}
\lim_{M\to\infty}\logE_{p\in\P_W\cap[M]} \1_{a\Z+b}(p)\, e( Q(p)\alpha) =0.
\end{equation}
\end{corollary}

\begin{proof}
In light of the prime number theorem in arithmetic progressions, we have for any bounded function $u\colon\N\to\C$ that
\[
\limsup_{M\to\infty}\Big|\logE_{p\in\P_W\cap[M]} u(p) -\logE_{p\in\P\cap[M]} W\1_{W\Z+1}(p)\, u(p)\Big|=0.
\]
It is then clear that \eqref{eqn_log_vinogradov_4} follows from \eqref{eqn_log_vinogradov}.
\end{proof}

\subsection{Preliminaries from ergodic theory}

For the reminder of this section, let $\Hilb$ be Hilbert space with inner product $\langle.,.\rangle$ and norm $\|.\|$. A \define{unitary operator} on $\Hilb$ is a bounded linear operator $U \colon \Hilb\to\Hilb$ which is invertible and preserves the inner product of the Hilbert space, that is, for all $f,g\in\Hilb$ we have
\[
\langle Uf,Ug\rangle=\langle f,g\rangle.
\]
We say that $\Hilb$ is the \define{orthogonal direct sum} of $\Hilb_1$ and $\Hilb_2$, and write $\Hilb = \Hilb_1 \oplus \Hilb_2$, if $\Hilb_1$ and $\Hilb_2$ are closed subspaces of $\Hilb$ satisfying:  
\begin{itemize}
    \item $\langle f_1 ,f_2\rangle=0$ whenever $f_1 \in \Hilb_1$ and $f_2 \in \Hilb_2$, and  
    \item for every $f \in \Hilb$, there exist $f_1 \in \Hilb_1$ and $f_2 \in \Hilb_2$ such that $f = f_1 + f_2$.  
\end{itemize}
Note that in this case, $f_1$ and $f_2$ are uniquely determined by $f$. In fact, $f_1$ equals the orthogonal projection of $f$ onto the subspace $\Hilb_1$, whereas $f_2$ is the orthogonal projection of $f$ onto $\Hilb_2$.

Let us now state von Neumann's mean ergodic theorem in an equivalent form.
\begin{theorem}[Mean ergodic theorem]
\label{thm_MET}
Let $U \colon \Hilb\to\Hilb$ be a unitary operator on a Hilbert space $\Hilb$. Then $\Hilb= \Hilb_{\mathrm{inv}}\oplus \Hilb_{\mathrm{erg}}$, where
\[
\Hilb_{\mathrm{inv}}=\{f\in\Hilb: Uf=f\}\qquad\text{and}\qquad
\Hilb_{\mathrm{erg}}=\Big\{f\in\Hilb: \lim_{H\to\infty}\Big\|\E_{h\in[H]}U^hf\Big\| = 0\Big\}.
\]
\end{theorem}

A proof of \cref{thm_MET} can be found in any standard book on ergodic theory, for example \cite[Theorem 1.2]{Petersen83}. We will use the following well-known corollary of \cref{thm_MET}.

\begin{corollary}[cf.~{\cite[p.~14]{Bergelson96}}]
\label{cor_rat_toterg_splitting}
Let $U \colon \Hilb\to\Hilb$ be a unitary operator on a Hilbert space $\Hilb$. Then
$\Hilb= \Hilb_{\mathrm{rat}}\oplus \Hilb_{\mathrm{tot\,erg}}$, where
\begin{align*}
\Hilb_{\mathrm{rat}}&=\overline{\big\{f\in\Hilb:~\exists q\in\N,~ U^{q}f=f\big\}},
\\
\Hilb_{\mathrm{tot\,erg}}&=\Big\{f\in\Hilb: \forall q\in \N,~\lim_{H\to\infty}\Big\|\E_{h\in[H]}U^{qh}f\Big\| = 0\Big\}.
\end{align*}
\end{corollary}

\begin{proof}
For every $q\in\N$, define
\[
\Hilb_{\mathrm{inv},q}=\{f\in\Hilb: U^qf=f\}\quad\text{and}\quad
\Hilb_{\mathrm{erg},q}=\Big\{f\in\Hilb: \lim_{H\to\infty}\Big\|\E_{h\in[H]}U^{qh}f\Big\| = 0\Big\}.
\]
By \cref{thm_MET}, we have $\Hilb= \Hilb_{\mathrm{inv},q}\oplus \Hilb_{\mathrm{erg},q}$ for all $q\in\N$. Therefore,
\[
\Hilb= \overline{\bigcup_{q\in\N} \Hilb_{\mathrm{inv},q}}\,\oplus\, \bigcap_{q\in\N} \Hilb_{\mathrm{erg},q}.
\]
It is now straightforward to verify that $\Hilb_{\mathrm{rat}}=\overline{\bigcup_{q\in\N} \Hilb_{\mathrm{inv},q}}$ and $\Hilb_{\mathrm{tot\,erg}}=\bigcap_{q\in\N} \Hilb_{\mathrm{erg},q}$, and the proof is complete.
\end{proof}

It is worth noting that, in ergodic theory, functions in $\Hilb_{\mathrm{tot,erg}}$ are often referred to as \define{totally ergodic} functions with respect to the operator~$U$.

\subsection{Preliminaries from Fourier analysis}

Throughout this paper, we identify the one-dimensional torus $\T = \R/\Z$ with the interval $[0,1)$, endowed with addition modulo $1$. This allows us to treat elements of $\T$ as real numbers in $[0,1)$, which simplifies notation. As above, we use the abbreviation $e(x) = e^{2\pi i x}$ for $x \in \R$.

Let $\mu$ be a finite Borel measure on $\T$. The \define{Fourier transform} of $\mu$ on $\T$ is the function $\hat{\mu}\colon \Z\to\C$ defined as 
\[
\hat\mu(m)=\int_{\T} e(mx)\d\mu(x),\qquad\forall m\in\Z.
\]
Since $\mu$ is a finite measure, we have $|\hat{\mu}(m)|\leq \hat{\mu}(0)=\mu(\T)<\infty$ for all $m\in\Z$, and hence $\hat{\mu}$ is a bounded function.
It is natural to ask what type of bounded functions on $\Z$ correspond to the Fourier transform of a finite Borel measure on $\T$. The answer to this question is provided by a classical result in Fourier analysis known as Herglotz's theorem. 

A function $\varphi\colon\Z\to\C$ is called \define{non-negative definite} if for all $M\in\N$ and $\lambda_1,\ldots,\lambda_M\in\C$ one has
\[
\sum_{i,j=1}^M \lambda_i\overline{\lambda_j}\varphi(i-j)\geq 0.
\]
Equivalently, $\varphi$ is \define{non-negative definite} if for any $M\in\N$ the matrix $A\in\C^{M\times M}$ defined by $A_{i,j}=\varphi(i-j)$ for all $1\leq i,j\leq M$ is a non-negative definite matrix.

\begin{theorem}[Herglotz's theorem]
\label{thm_herglotz}
Let $\varphi\colon\Z\to\C$ be bounded. Then $\varphi$ is non-negative definite if and only if there exists a finite Borel measure $\mu$ on $\T=\R/\Z$ such that $\varphi=\hat{\mu}$, or in other words,
\[
\varphi(m)=\int_{\T} e(mx)\d\mu(x),\qquad\forall m\in\Z.
\]
\end{theorem}

If $\A\subset\ell^\infty(\Z)$ is a translation invariant $C^*$-subalgebra that admits logarithmic averages along an increasing sequence $\bN=(N_s)_{s\in\N}$ (see \cref{def_algebra_properties}), then it is straightforward to check that for any $f\in\A$ the function
\[
m\mapsto \logE_{n\in\bN} f(n)\overline{f(n+m)}
\]
is non-negative definite. By Herglotz's theorem, there exists a finite Borel measure $\mu_f$ on $\T$ such that
\begin{equation}
\label{eqn_spectral_measure}
\logE_{n\in\bN} f(n)\overline{f(n+m)}=\hat\mu_f(m)=\int_{\T} e(mx)\d\mu_f(x).
\end{equation}
We call $\mu_f$ the \define{spectral measure of $f$}.

\begin{lemma}
\label{lem_spectral_measure_isometry_property}
Let $N_1<N_2<\ldots\in\N$, and let $\A\subset\ell^\infty(\Z)$ be an affinely invariant $C^*$-subalgebra that admits logarithmic averages along $\bN=([N_s])_{s\in\N}$. Then for any $\ell\in\N$, $m_1,\ldots,m_\ell\in\Z$, $c_1,\ldots, c_\ell\in\C$, and $f\in\A$ we have
\[
\logE_{n\in\bN} \bigg| \sum_{i=1}^\ell c_i f(n+m_i)\bigg|^2
=
\int_\T\bigg| \sum_{i=1}^\ell c_i e(m_ix)\bigg|^2\d\mu_f(x),
\]
where $\mu_f$ denotes the spectral measure of $f$.
\end{lemma}

\begin{proof}
By expanding the square and rearranging, we can rewrite the left hand side as
\begin{align*}
\logE_{n\in\bN} \bigg| \sum_{i=1}^\ell c_i f(n+m_i)\bigg|^2
&=\sum_{i,j=1}^\ell c_i\overline{c_j}\, \logE_{n\in\bN}
f(n+m_i)\overline{f(n+m_j)}
\\
&=\sum_{i,j=1}^\ell c_i\overline{c_j}\, \logE_{n\in\bN} f(n)\overline{f(n+m_j-m_i)}.
\end{align*}
Let $\mu_f$ denote the spectral measure of $f$.
If we now use \eqref{eqn_spectral_measure}, reorder, and collapse the resulting expression back into a square, we obtain
\begin{align*}
\sum_{i,j=1}^\ell c_i\overline{c_j}\, \logE_{n\in\bN} f(n)\overline{f(n+m_j-m_i)}
&=\sum_{i,j=1}^\ell c_i\overline{c_j} \int_\T e((m_j-m_i)x)\d\mu_f(x)
\\
&=\int_\T\bigg| \sum_{i=1}^\ell c_i e(m_ix)\bigg|^2\d\mu_f(x).
\end{align*}
This completes the proof.
\end{proof}

\section{The structure theorem}

Many proofs in combinatorics depend on some type of structure theorem that allows one to decompose a given objet (such as a set, function, or graph) into two components, a ``structured'' component that is responsible for the main behavior and can be analyzed using combinatorial methods, and a ``random'' component that represents noise and can be dealt with using probabilistic techniques; see \cite{Tao07b} for a survey.
The proofs of our main theorems also follow this strategy, and the purpose of this section is to introduce the structure theorem on which our arguments rely.

\subsection{Statement of the structure theorem}

\begin{definition}
Let $\bN=(N_s)_{s\in\N}$ with $N_1<N_2<\ldots\in\N$, and let $\A\subset\ell^\infty(\Z)$ be a translation invariant $C^*$-subalgebra that admits logarithmic averages along $\bN$.
\begin{itemize}
\item 
We say $f\in\A$ is \define{locally totally ergodic (with respect to logarithmic averages along~$\bN$)} if for all $q\in\Z\setminus\{0\}$ we have
\[
\lim_{H\to\infty}\logE_{n\in \bN}\Big|\E^{\phantom{\log{}}}_{h\in [H]}\, f(n+qh)\Big|^2=0.
\]   
\item 
We say $f\in\A$ is \define{locally rationally almost periodic (with respect to logarithmic averages along~$\bN$)} if for all $\epsilon>0$ there exists some $q\in\N$ such that 
\[
\sup_{m\in\N}\Big(\logE_{n\in\bN}\big|f(n+qm)-f(n)\big|^2\Big)\leq \epsilon.
\]
\end{itemize}
\end{definition}

Note that the collection of all locally totally ergodic elements of $\A$ forms a closed subspace of $\A$ that is invariant under complex conjugation, whereas the collection of all locally rationally almost periodic elements of $\A$ forms a $C^*$-subalgebra of $\A$. Additional convenient properties of locally rationally almost periodic and locally totally ergodic functions are collected in \cref{lem_basic_properties_locally_periodic_locally_aperiodic} at the end of this subsection, and in \cref{sec_spectral_char_locally_periodic_locally_aperiodic}.

\begin{theorem}[Structure Theorem]
\label{thm_str}
Let $\bN=(N_s)_{s\in\N}$ with $N_1<N_2<\ldots\in\N$, and let $\A\subset\ell^\infty(\Z)$ be a separable and translation invariant $C^*$-subalgebra that admits logarithmic averages along $\bN$.
Then there exists a separable and translation invariant $C^*$-subalgebra $\A'\subset\ell^\infty(\Z)$ that admits logarithmic averages along $\bN$, contains $\A$ as a subset, and such that the following holds:
For all $f\in\A'$ there exist $f_{\str},f_{\rnd}\in \A'$ with
\[
f=f_{\str}+f_{\rnd},
\]
where $f_{\str}$ is locally rationally almost periodic with respect to logarithmic averages along $\bN$ and $f_{\rnd}$ is locally totally ergodic with respect to logarithmic averages along $\bN$. Moreover, if $f$ takes values in a closed interval $[a,b]\subset\R$ then the same holds for $f_{\str}$.
\end{theorem}

\cref{thm_str} can be viewed as an integer analogue of \cref{cor_rat_toterg_splitting}. Similar structure theorems in the integers were obtained in \cite{MRR19}.
The proof of \cref{thm_str} is given in \cref{sec_proof_str_thm}.

\begin{remark}
It will be clear from the construction used in the proof of \cref{thm_str} that
if $\A$ is affinely invariant then the $C^*$-supalgebra $\A'$ guaranteed by \cref{thm_str} is also affinely invariant.
Similarly, if $\bPW=(\P_W\cap[M_{t_k}])_{t\in\N}$, where $M_1<M_2<\ldots\in\N$ and $W\in\N$, and $\A$ admits iterated affine correlations along $\bPW$ (see~\cref{def_algebra_properties}), then so does $\A'$.
\end{remark}

\begin{lemma}[Basic properties of locally rationally almost periodic and locally totally ergodic functions]
\label{lem_basic_properties_locally_periodic_locally_aperiodic}
Let $\bN=(N_s)_{s\in\N}$ with $N_1<N_2<\ldots\in\N$, and let $\A\subset\ell^\infty(\Z)$ be a translation invariant $C^*$-subalgebra that admits logarithmic averages along $\bN$.
Suppose $f\in\A$ is locally rationally almost periodic and $g\in\A$ is locally totally ergodic. Then:
\begin{enumerate}
[label=(\roman{enumi}),ref=(\roman{enumi}),leftmargin=*]
\item
\label{itm_basic_properties_locally_periodic_locally_aperiodic_i}
$f$ and $g$ are orthogonal, i.e., $\E^{\log}_{n\in\bN} f(n)\overline{g(n)}=0$;
\item
\label{itm_basic_properties_locally_periodic_locally_aperiodic_ii}
For all $b\in\Z$, $f\circ\tau^b$ is locally rationally almost periodic and $g\circ\tau^b$ is locally totally ergodic;
\item
\label{itm_basic_properties_locally_periodic_locally_aperiodic_iii}
If $\A$ is affinely invariant then for all $a\in\N$, $f\circ\sigma_a$ is locally rationally almost periodic and $g\circ\sigma_a$ is locally totally ergodic.
\end{enumerate}
\end{lemma}

Each of the parts \ref{itm_basic_properties_locally_periodic_locally_aperiodic_i}--\ref{itm_basic_properties_locally_periodic_locally_aperiodic_iii} is straightforward to verify using the definition of locally rationally almost periodic and locally totally ergodic functions, and we therefore omit their proofs.

\subsection{Outlining the proofs of Theorems~\ref{thm_1_tech} and~\ref{thm_2_tech}}
\label{sec_outline_of_proof}

We now explain how the structure theorem provides a useful framework for proving both Theorems~\ref{thm_1_tech} and~\ref{thm_2_tech}.
To albeit with the hypothesis in Theorems~\ref{thm_1_tech} and~\ref{thm_2_tech}, let us assume for the reminder of this section that $N_1<N_2<\ldots\in\N$ and $M_1<M_2<\ldots\in\N$ are two increasing integer sequences, and $\A\subset\ell^\infty(\Z)$ is a separable, affinely invariant $C^*$-subalgebra that admits logarithmic averages along $\bN=([N_s])_{s\in\N}$ and iterated affine correlations along $\bPW=(\P_W\cap [M_t])_{t\in\N}$ for all $W\in\N$.

By replacing $\A$ with the $C^*$-subalgebra $\A'\subset\ell^\infty(\Z)$ guaranteed by \cref{thm_str}, we can assume without loss of generality that $\1_A\in\A$ splits into two components,
\begin{equation}
\label{eqn_splitting_of_A}
\1_A=f_{\str}+f_{\rnd},
\end{equation}
where $f_{\str}\in\A$ is locally rationally almost periodic and $f_{\rnd}\in\A$ is locally totally ergodic. Also, since $\1_A$ is non-negative, it follows that the function $f_{\str}$ is non-negative too, which will be important for technical reasons later on.

The next theorem lets us conclude that the random component $f_{\rnd}$ in the decomposition $\1_A=f_{\str}+f_{\rnd}$ does not contribute to the averages in \eqref{eqn_1_tech} and \eqref{eqn_2_tech}. This will allow us to ``ignore'' $f_{\rnd}$ and focus solely on $f_{\str}$ in the proofs of Theorems~\ref{thm_1_tech} and~\ref{thm_2_tech}.

\begin{theorem}[Handling the ``random'' component]
\label{thm_rnd_component}
If $f\in\A$ is locally totally ergodic with respect to logarithmic averages along $\bN$, then for all $k,W,a\in\N$ with $k\geq 2$, all $g\in\A$, and all non-constant polynomials $Q$ with integer coefficients we have
\[
\logE_{p\in\bPW^{*k}}\logE_{n\in\bN} f\big(an+Q(p)\big)g(pn)=0.
\]
\end{theorem}

A proof of \cref{thm_rnd_component} is given in \cref{sec_proof_of_thm_rnd_component}.

The next theorem can be viewed as a (somewhat technical) generalization of a
theorem of Ahlswede, Khachatrian, and \Sarkozy{} \cite{AKS99}, which in turn was a generalization of the aforementioned result of Davenport and \Erdos{} \cite{DE36,DE51}.

\begin{theorem}[Handling the ``structured'' component]
\label{thm_str_component}
Suppose $a_1,a_2,\ldots\in\N$ and $f\in\A$ with $f(n)\geq 0$ for all $n\in\N$.
Then for any $\epsilon>0$ and any $W\in\N$ there are $k\geq 2$ and an infinite set $I\subset\N$, such that for all $i,j\in I$ with $i<j$ we have
\begin{equation*}
\logE_{p\in\bPW^{*k}} \logE_{n\in \bN} f(a_in)f(a_jpn) \geq  \Big(\limsup_{j\to\infty}\logE_{n\in\bN} f(a_jn)\Big)^2-\epsilon.    
\end{equation*}
\end{theorem}

A proof of \cref{thm_str_component} can be found in \cref{sec_proof_of_thm_str_component}.
We now explain why Theorems~\ref{thm_rnd_component} and~\ref{thm_str_component} are enough to derive Theorems~\ref{thm_1_tech} and~\ref{thm_2_tech}.

\begin{proof}[Proof of \cref{thm_1_tech} assuming Theorems~\ref{thm_rnd_component} and~\ref{thm_str_component}]
Fix $\epsilon>0$. Or goal is to find $W\geq 1$ and $k\geq 2$ such that
\begin{equation}
\label{eqn_1_tech_1}
\logE_{p\in\bPW^{*k}} \logE_{n\in \bN} \1_A(n+Q(p))\1_A(pn)\geq \Big(\logE_{n\in\bN} \1_A(n)\Big)^2-\epsilon.
\end{equation}
If $Q$ is a constant polynomial then the condition $Q(1)=0$ implies that $Q=0$. In this case, 
\eqref{eqn_1_tech_1} follows from \cref{thm_str_component} applied with $f(n)=1_A(n)$ and $a_1=a_2=\ldots=1$.

If $Q$ is non-constant then we proceed as follows. 
Using the decomposition \eqref{eqn_splitting_of_A}, we can rewrite the left hand side as
\begin{align*}
\logE_{p\in\bPW^{*k}} & \logE_{n\in \bN} \1_A(n+Q(p))\1_A(pn) 
\\ 
&=\logE_{p\in\bPW^{*k}} \logE_{n\in \bN} f_{\str}(n+Q(p))f_{\str}(pn)+\overbrace{\logE_{p\in\bPW^{*k}} \logE_{n\in \bN}f_{\str}(n+Q(p))f_{\rnd}(pn)}^{[1]}
\\
&\quad+\underbrace{\logE_{p\in\bPW^{*k}} \logE_{n\in \bN}f_{\rnd}(n+Q(p))f_{\str}(pn)}_{[2]}+\underbrace{\logE_{p\in\bPW^{*k}} \logE_{n\in \bN} f_{\rnd}(n+Q(p))f_{\rnd}(pn)}_{[3]}.
\end{align*}
By \cref{lem_basic_properties_locally_periodic_locally_aperiodic}, we have for all $p$,
\[
\logE_{n\in \bN}f_{\str}(n+Q(p))f_{\rnd}(pn)=0\qquad\text{and}
\qquad \logE_{n\in \bN}f_{\rnd}(n+Q(p))f_{\str}(pn)=0.
\]
This shows that the cross terms {\small[1]} and {\small[2]} disappear, and only the diagonal terms remain. Moreover, in light of \cref{thm_rnd_component} (applied with $f(n)=f_{\rnd}(n)$, $g(n)=f_{\rnd}(n)$, and $a=1$), the final term {\small[3]} also does not contribute to the average for any $W,k\in\N$ as long as $k\geq 2$. This proves that
\begin{align*}
\logE_{p\in\bPW^{*k}} \logE_{n\in \bN} \1_A(n+Q(p))\1_A(pn) 
=\logE_{p\in\bPW^{*k}} \logE_{n\in \bN} f_{\str}(n+Q(p))f_{\str}(pn).
\end{align*}
Since $f_{\str}$ is locally rationally almost periodic, there exists $W\in\N$ such that
\[
\sup_{m'\in\N}\Big(\logE_{n\in\bN}\big|f_{\str}(n+Wm')-f_{\str}(n)\big|^2\Big)\leq \frac{\epsilon^2}{4}.
\]
Using $Q(1)=0$, we see that for every $m\in\N$ there is $m'\in\N$ such that $Q(Wm+1)=m'W$. 
Since any $p=p_1\cdots p_k$, where $p_1,\ldots,p_k\in\P_W$, can be written in the form $Wm+1$ for some $m\in\N$, it follows that
\[
\logE_{p\in\bPW^{*k}}\logE_{n\in\bN}\big|f_{\str}(n+Q(p))-f_{\str}(n)\big|^2\leq \frac{\epsilon^2}{4}.
\]
Therefore, using the Cauchy-Schwarz inequality, we get
\begin{equation}
\label{eqn_1_tech_2}
\logE_{p\in\bPW^{*k}} \logE_{n\in \bN} \1_A(n+Q(p))\1_A(pn) 
\geq \logE_{p\in\bPW^{*k}} \logE_{n\in \bN} f_{\str}(n)f_{\str}(pn)-\frac{\epsilon}{2}.
\end{equation}
Finally, note that $\E^{\log}_{n\in\bN} f_{\rnd}(n)=0$ and hence $\E^{\log}_{n\in\bN}f_{\str}(n)=\E^{\log}_{n\in\bN}\1_A(n)$.
So it follows from \cref{thm_str_component}, applied with $f(n)=f_{\str}(n)$ and $a_1=a_2=\ldots=1$, that there is some $k\geq 2$ for which we have
\begin{equation}
\label{eqn_1_tech_3}
\logE_{p\in\bPW^{*k}} \logE_{n\in \bN} f_{\str}(n)f_{\str}(pn)\geq \Big(\logE_{n\in\bN} f_{\str}(n)\Big)^2-\frac{\epsilon}{2}= \Big(\logE_{n\in\bN} \1_A(n)\Big)^2-\frac{\epsilon}{2}.
\end{equation}
Since \eqref{eqn_1_tech_2} and \eqref{eqn_1_tech_3} combined imply \eqref{eqn_1_tech_1}, the proof is finished.
\end{proof}

\begin{proof}[Proof of \cref{thm_2_tech} assuming Theorems~\ref{thm_rnd_component} and~\ref{thm_str_component}]
We use a similar argument the one used above to show that Theorems~\ref{thm_rnd_component} and~\ref{thm_str_component} imply \cref{thm_1_tech}.
Fix $\epsilon>0$ and let $a_1,a_2,\ldots\in\N$ be a divisible sequence. Our goal is to find $k\geq 2$ and $\{u,vu\}\subset \{a_j:j\in\N\}$ such that
\begin{equation}
\label{eqn_2_tech_1}
\logE_{p\in\bP^{*k}}\logE_{n\in\bN} \1_A(un+Q(vp))\1_A(uvpn) \geq \Big(\limsup_{j\to\infty}\logE_{n\in\bN} \1_A(a_jn)\Big)^2-\epsilon.
\end{equation}
By replacing $(a_j)_{j\in\N}$ with a subsequence of itself if necessary, we can assume without loss of generality that $a_{j+1}$ is a multiple of $a_j$ for all $j\in\N$. This can be done in a way that leaves $\limsup_{j\to\infty}\E^{\log}_{n\in\bN} \1_A(a_jn)$ unchanged.

If $Q$ is constant then $Q(0)=0$ implies $Q=0$. If this is the case then \eqref{eqn_2_tech_1} is a consequence of \cref{thm_str_component}, taking $u=a_i$ and $v=\frac{a_j}{a_i}$.
Therefore, we can assume for the remainder of this proof that $Q$ is non-constant.

Arguing as above, we can use \eqref{eqn_splitting_of_A} to split $\1_A$ into  a ``structured'' component $f_{\str}$ and a ``random'' component $f_{\rnd}$. Consequently, the left hand side of \eqref{eqn_2_tech_1} splits into four terms. The cross terms involving one ``structured'' component $f_{\str}$ and one ``random'' component $f_{\rnd}$ are $0$ due to orthogonality, see \cref{lem_basic_properties_locally_periodic_locally_aperiodic}. Moreover, the term involving twice the ``random'' component $f_{\rnd}$ is $0$ because of \cref{thm_rnd_component} (applied with $f(n)=f_{\rnd}(n)$, $g(n)=f_{\rnd}(a_jn)$, $W=1$, $a=a_i$, and $b=\frac{a_j}{a_i}$). This gives that
\begin{equation}
\label{eqn_2_tech_4}
\logE_{p\in\bP^{*k}} \logE_{n\in \bN} \1_A\big(a_in+Q\big(\tfrac{a_jp}{a_i}\big)\big)\1_A(a_jpn) 
=\logE_{p\in\bP^{*k}} \logE_{n\in \bN} f_{\str}\big(a_in+Q\big(\tfrac{a_jp}{a_i}\big)\big)f_{\str}(a_jpn).
\end{equation}
Invoking \cref{thm_str_component} (with $f(n)=f_{\str}(n)$ and $W=1$), we can find $k\geq 2$ and an infinite set $I\subset\N$, such that for all $i,j\in I$ with $i<j$ we have
\[
\logE_{p\in\bP^{*k}}\logE_{n\in\bN} f_{\str}\big(a_in\big)f_{\str}(a_jpn) \geq  \Big(\limsup_{j\to\infty}\logE_{n\in\bN} f_{\str}(a_jn)\Big)^2-\frac{\epsilon}{2}.
\]
Since $\E^{\log}_{n\in\bN}f_{\str}(bn)=\E^{\log}_{n\in\bN}\1_A(bn)$ for all $b\in\N$, this implies that
\begin{equation}
\label{eqn_2_tech_2}
\logE_{p\in\bP^{*k}}\logE_{n\in\bN} f_{\str}\big(a_in\big)f_{\str}(a_jpn) \geq  \Big(\limsup_{j\to\infty}\logE_{n\in\bN} \1_A(a_jn)\Big)^2-\frac{\epsilon}{2}.
\end{equation}

Let $i\in I$ be fixed for the remainder of this proof.
Using that $f_{\str}$ is locally rationally almost periodic, there exists $q\in\N$ such that
\[
\sup_{m\in\N}\Big(\logE_{n\in\bN}\big|f_{\str}(a_in+qm)-f_{\str}(a_in)\big|^2\Big)\leq \frac{\epsilon^2}{4}.
\]
Since $(a_j)_{j\in \N}$ is divisible, there exist $j\in I$ such that $a_j$ is a multiple of $qa_i$.
Since $Q(0)=0$, for any $p$ the number $Q(\frac{a_jp}{a_i})$ is a multiple of $q$.
It follows that for all $p$ we have
\begin{equation}
\label{eqn_2_tech_5}
\logE_{n\in\bN}\big|f_{\str}\big(a_in+Q\big(\tfrac{a_jp}{a_i}\big)\big)-f_{\str}(a_in)\big|^2\leq \frac{\epsilon^2}{4}.
\end{equation}
Applying \eqref{eqn_2_tech_5} and the Cauchy-Schwarz inequality to \eqref{eqn_2_tech_4}, we can conclude that
\begin{equation}
\label{eqn_2_tech_3}
\logE_{p\in\bP^{*k}} \logE_{n\in \bN} \1_A\big(a_in+Q\big(\tfrac{a_jp}{a_i}\big)\big)\1_A(a_jpn) 
\geq \logE_{p\in\bP^{*k}} \logE_{n\in \bN} f_{\str}(a_in)f_{\str}(a_jpn)-\frac{\epsilon}{2}.
\end{equation}
Since \eqref{eqn_2_tech_2} and \eqref{eqn_2_tech_3} give \eqref{eqn_2_tech_1}, we can take $u=a_i$ and $v=\frac{a_j}{a_i}$ and the proof is complete.
\end{proof}

\section{Proof of the structure theorem}
\label{sec_proof_str_thm}

The purpose of this section is to prove \cref{thm_str}.
As noted previously, it is natural to interpret \cref{thm_str} as a discrete analogue of \cref{cor_rat_toterg_splitting}, and the main idea behind its proof is to transfer \cref{cor_rat_toterg_splitting} from the setting of unitary operators on Hilbert spaces to the setting of the translation operator acting on bounded functions over the integers.

\subsection{The Hilbert space $\mathcal{L}^2(\A, \bN)$}
Let $\bN=(N_s)_{s\in\N}$ with $N_1<N_2<\ldots\in\N$, and let $\A\subset\ell^\infty(\Z)$ be a $C^*$-subalgebra that admits logarithmic averages along $\bN$.
Consider the space
\begin{equation*}
\mathscr{L}^2(\A, \bN) = \left\{ f : \Z \to \C : \forall \epsilon > 0,~\exists g \in \A, \limsup_{s \to \infty} \logE_{n\in[N_s]}|f(n) - g(n)|^2 < \epsilon \right\}.
\end{equation*}
We can define an equivalence relation on $\mathscr{L}^2(\A, \bN)$ via
\[
f \sim g\quad\iff\quad \logE_{n \in \bN} |f(n) - g(n)|^2 = 0.
\]
For any function $f\in \mathscr{L}^2(\A,\bN)$, let $[f]_\sim$ denote the equivalence class of $f$ with respect to $\sim$, and let
\[
\mathcal{L}^2(\A,\bN)=\mathscr{L}^2(\A, \bN)/\sim
\]
denote the quotient space of $\mathscr{L}^2(\A,\bN)$ by $\sim$.
Note that pointwise addition, pointwise multiplication, scalar multiplication, and complex conjugation on $\mathscr{L}^2(\A, \bN)$ factor through to $\mathcal{L}^2(\A,\bN)$ via
\begin{align*}
[f]_\sim+[g]_\sim &=[f+g]_\sim,
\\
[f]_\sim \cdot [g]_\sim &=[f\cdot g]_\sim,
\\
c[f]_\sim&=[cf]_\sim,
\\
\overline{[f]_\sim}&=\big[\overline{f}\big]_\sim,
\end{align*}
for all $f,g\in \mathscr{L}^2(\A,\bN)$ and $c\in\C$.
This shows that $\mathcal{L}^2(\A, \bN)$ is a vector space over $\C$.
If one considers \Cesaro{} averages instead of logarithmic averages, it was shown by Farhangi in \cite[Theorem 2.1]{Farhangi24} that $\mathcal{L}^2(\A, \bN)$ is actually a Hilbert space.
It is straightforward to adapt Farhangi's argument to our situation; for completeness we include a proof below. 

\begin{theorem}
\label{thm_farhangi}
    The space $\mathcal{L}^2(\A, \bN)$ is a Hilbert space with inner product
\begin{equation*}
\langle [f]_\sim,[g]_\sim\rangle = \logE_{n \in \bN} f(n)\overline{g(n)}.
\end{equation*}
\end{theorem}

\begin{proof}
It is clear that $\mathcal{L}^2(\A,\bN)$ is an inner product space; to prove that it is a Hilbert space it therefore remains to verify that it is complete.
For the remainder of this proof we abuse notation and identify each function in $\mathcal{L}^2(\A,\bN)$  with an arbitrarily chosen representative in its equivalency class in $\mathscr{L}^2(\A, \bN)$.
Let $(f_i)_{i \in \N}$ be a Cauchy sequence in $\mathcal{L}^2(\A,\bN)$. Choose $\epsilon_i \to 0$ such that for every $i,j\in\N$,
\[
\logE_{n \in \bN}|f_j(n) - f_i(n)|^2 \leq \epsilon_{\min\{i,j\}}.
\]
This means for any $i,j\in\N$ there exists some $s_{i,j}\in\N$ such that for every $s\geq s_{i,j}$ we have
\[
\logE_{n \in [N_s]}|f_j(n) - f_i(n)|^2 \leq  2\epsilon_{\min\{i,j\}}.
\]
Now we define $s_0=0$ and $s_j=\max\{s_{1,j},s_{2,j},\ldots,s_{j,j}\}$ for $j\geq 1$.
Then for all $i,j\in\N$ with $i\leq j$ and all $s\geq s_j$ we have
\begin{equation}
\label{eqn_SF_1}
\logE_{n \in [N_s]}|f_j(n) - f_i(n)|^2 \leq  2\epsilon_{i}.
\end{equation}
By replacing $(s_j)$ with a subsequence of itself if necessary, we can assume that $s_j$ is sufficiently larger relative to $s_{j-1}$ so that
\begin{equation}
\label{eqn_SF_2}
\Big(\max_{n\in[N_{s_{j-1}}]}\max_{\ell\leq j-1}|f_\ell(n)|^2\Big)\frac{\sum_{n\in [N_{s_{j-1}}]}\frac{1}{n}}{\sum_{n\in[N_{s_j}]}\frac{1}{n}}\to 0~\text{as}~j\to\infty.
\end{equation}
Define
\[
f(n)=\sum_{i\in\N} \1_{(N_{s_{i-1}},N_{s_i}]}(n) f_i(n).
\]
Then for any $i\in\N$ and $s_{j-1}<s\leq s_{j}$ we have
\begin{align*}
 \logE_{n \in [N_s]} |f(n) - f_i(n)|^2
&=\sum_{\ell=1}^j \logE_{n \in [N_s]}\1_{(N_{s_{\ell-1}},N_{s_\ell}]}(n)|f(n) - f_i(n)|^2
\\
&=\sum_{\ell=1}^j \logE_{n \in [N_s]}\1_{(N_{s_{\ell-1}},N_{s_\ell}]}(n)|f_\ell(n) - f_i(n)|^2
\\
&= [1]+[2]+[3],
\end{align*}
where
\begin{align*}
[1]=\sum_{\ell=1}^{j-2} \logE_{n \in [N_s]}\1_{(N_{s_{\ell-1}},N_{s_\ell}]}(n)|f_\ell(n) - f_i(n)|^2,
\\
[2]=\logE_{n \in [N_s]}\1_{(N_{s_{j-2}},N_{s_{j-1}}]}(n)|f_{j-1}(n) - f_i(n)|^2,
\\
[3]=\logE_{n \in [N_s]}\1_{(N_{s_{j-1}},N_{s_j}]}(n)|f_{j}(n) - f_i(n)|^2.
\end{align*}
It follows from \eqref{eqn_SF_1} that $[2]\leq 2\epsilon_i$ as well as $[3]\leq 2\epsilon_i$. For $[1]$, we have the estimate
\begin{align*}
[1]&\leq 2\Big(\max_{n\in[N_{s_{j-2}}]}\max_{\ell\leq j-2}|f_\ell(n)|^2\Big)\logE_{n \in [N_s]}\1_{[1,N_{s_{j-2}}]}(n),
\\
&\leq 2\Big(\max_{n\in[N_{s_{j-2}}]}\max_{\ell\leq j-2}|f_\ell(n)|^2\Big)\frac{\sum_{n\in [N_{s_{j-2}}]}\frac{1}{n}}{\sum_{n\in[N_{s_{j-1}}]}\frac{1}{n}}.
\end{align*}
Therefore, \eqref{eqn_SF_2} implies that $[1]=\oh_{j\to\infty}(1)$. Combining all of the above shows
\[
\limsup_{s\to\infty}\logE_{n \in [N_s]} |f(n) - f_i(n)|^2\leq 4\epsilon_i,
\]
which implies that the Cauchy sequence $(f_i)_{i\in\N}$ converges to $f$. This proves that $\mathcal{L}^2(\A,\bN)$ is complete.
\end{proof}

\subsection{Proof of \cref{thm_str}}

\begin{proof}[Proof of \cref{thm_str}]
Let $\mathscr{L}^2(\A,\bN)$ and $\mathcal{L}^2(\A,\bN)$ be as defined above, and define $\A'=\mathscr{L}^2(\A,\bN)\cap\ell^\infty(\Z)$. Note that if $\A$ is separable (resp.~translation invariant/dilation invariant/admits iterated affine correlations along $\bPW$) then $\A'$ has the same property. Moreover, due to \cref{thm_farhangi}, $\mathcal{L}^2(\A,\bN)$ is a Hilbert space with inner product $\langle [f]_\sim, [g]_\sim\rangle=\E^{\log}_{n\in\bN} f(n)\overline{g(n)}$ and norm $\|[f]_\sim\|=(\E^{\log}_{n\in\bN}|f(n)|^2)^{1/2}$.
Define an operator $U\colon \mathcal{L}^2(\A,\bN)\to \mathcal{L}^2(\A,\bN)$ via
\[
U[f]_\sim=[f\circ \tau]_\sim,\qquad\forall [f]_\sim\in \mathcal{L}^2(\A,\bN),
\]
where $\tau(n)=n+1$ is the shift map defined in \eqref{eqn_def_dilation_translation}.
Since logarithmic averages are shift invariant, we see that $\langle U[f]_\sim,U[g]_\sim\rangle =\langle [f]_\sim,[g]_\sim\rangle$. In other words, $U$ is a unitary operator on $\mathcal{L}^2(\A,\bN)$. 
It follows from \cref{cor_rat_toterg_splitting} that 
$\mathcal{L}^2(\A,\bN)=\Hilb_{\mathrm{rat}}\oplus \Hilb_{\mathrm{tot\,erg}}$, where $\Hilb_{\mathrm{rat}}$ and $\Hilb_{\mathrm{tot\,erg}}$ are as defined in the statement of \cref{cor_rat_toterg_splitting}.
It is now easy to check that $[f]_\sim\in \Hilb_{\mathrm{rat}}$ if and only if $f$ is locally rationally almost periodic with respect to logarithmic averages along $\bN$, and similarly $[f]_\sim\in\Hilb_{\mathrm{tot\,erg}}$ if and only if $f$ is locally totally ergodic with respect to logarithmic averages along $\bN$. It follows that for any $f\in\A'$ we can find 
$f_{\str},f_{\rnd}\colon \N\to\C$ such that 
$[f_{\str}]_{\sim}\in \Hilb_{\mathrm{rat}}$, $[f_{\rnd}]_{\sim}\in \Hilb_{\mathrm{tot\,erg}}$, and 
\[
f=f_{\str}+f_{\rnd}.
\]
It remains to show that we can take $f_{\str}$ and $f_{\rnd}$ to be bounded, or equivalently, that the equivalency classes of $f_{\str}$ and $f_{\rnd}$ contain representatives that belong to $\ell^\infty(\Z)$.

Suppose $f$ takes values in the closed interval $[a,b]$ for some $a\leq b\in\R$. If we truncate $f_{\str}$ above by $a$ and below by $b$, then the resulting function, let us call if $f_{\str}'$, also has the property that  $[f_{\str}']_\sim\in \Hilb_{\mathrm{rat}}$. Since $[f_{\str}]_\sim$ coincides with the orthogonal projection of $[f]_\sim$ onto $\Hilb_{\mathrm{rat}}$, we know that $[f_{\str}]_\sim$ is the unique element in $\Hilb_{\mathrm{rat}}$ that minimizes the distance between $[f]_\sim$ and $\Hilb_{\mathrm{rat}}$ within the Hilbert space $\mathcal{L}^2(\A,\bN)$. But the distance between $f_{\str}'$ and $f$ is no larger than the distance between $f_{\str}$ and $f$. By uniqueness, we conclude that $[f_{\str}']_\sim=[f_{\str}]_\sim$. 
This shows that with out loss of generality, we can assume that $f_{\str}$ takes values in $[a,b]$. In particular, both $f_{\str}$ and $f_{\rnd}$ are bounded and hence $f_{\str},f_{rnd}\in\A'$, completing the proof. 
\end{proof}

\section{Controlling the ``random'' component}

The goal of this section is to provide a proof of \cref{thm_rnd_component}. This proof has two main ingredients, a spectral characterization of locally totally ergodic functions given in \cref{sec_spectral_char_locally_periodic_locally_aperiodic}, and a new multiplicative analogue of van der Corput's inequality proved in \cref{sec_multiplicative_vdC}.

\subsection{Spectral characterizations of locally rationally almost periodic and locally totally ergodic functions}
\label{sec_spectral_char_locally_periodic_locally_aperiodic}

Recall that we identify the torus $\T=\R/\Z$ with the interval $[0,1)$ in the natural way.
In particular, we call an element of $\T$ \define{rational} if it corresponds to a rational number in $[0,1)$, and \define{irrational} otherwise.
By abuse of language, we say that a finite Borel measure on $\T$ is supported on rational (resp.~irrational) numbers if the set of rational (resp.~irrational) numbers in $\T$ has full measure. 
\begin{theorem}
\label{thm_spectral_char_per_aper}
Let $\bN=(N_s)_{s\in\N}$ with $N_1<N_2<\ldots\in\N$, and let $\A\subset\ell^\infty(\Z)$ be a translation invariant $C^*$-subalgebra that admits logarithmic averages along $\bN$.
Let $f\in\A$. The following hold:
\begin{enumerate}
[label=(\roman{enumi}),ref=(\roman{enumi}),leftmargin=*]
\item
\label{itm_spectral_char_per_aper_i}
if $f$ is locally rationally almost periodic then its spectral measure $\mu_f$ is supported on rational numbers.
\item
\label{itm_spectral_char_per_aper_ii}
if $f$ is locally totally ergodic then its spectral measure $\mu_f$ is supported on irrational numbers.
\end{enumerate}
\end{theorem}

\begin{proof}
Suppose $f\in\ell^\infty(\Z)$, and let $\mu_f$ denote its spectral measure.
Using the dominated convergence theorem, we obtain for every $q\in\N$ that
\begin{align*}
\lim_{M\to\infty} \E_{m\in[M]}\hat\mu_f(qm)
&=\lim_{M\to\infty} \E_{m\in[M]}\int_{\T} e(qmx)\d\mu_f(x)  
\\
&=\int_{\T} \bigg( \lim_{M\to\infty} \E_{m\in[M]}(e(qmx)\bigg)\d\mu_f(x) 
\\
&=\int_{\T} \1_{\big\{0,\frac{1}{q},\ldots,\frac{q-1}{q}\big\}}(x)\d\mu_f(x).
\end{align*}
Therefore, we have for all $q\in\N$,
\begin{equation}
\label{eqn_spectral_char_per_aper_1}
\mu_f\big(\big\{0,\tfrac{1}{q},\ldots,\tfrac{q-1}{q}\big\}\big)= \lim_{M\to\infty} \E_{m\in[M]}\hat\mu_f(qm).
\end{equation}

We will use \eqref{eqn_spectral_char_per_aper_1} to prove both \ref{itm_spectral_char_per_aper_i} and \ref{itm_spectral_char_per_aper_ii}.
Let us first show \ref{itm_spectral_char_per_aper_i}. If $f$ is locally rationally almost periodic then, by definition, for all $\epsilon>0$ there exists some $q\in\N$ such that
\[
\sup_{m\in\N}|\hat\mu_f(qm)-\hat\mu_f(0)|<\epsilon.
\]
This implies that
\[
\mu_f(\T)=\hat\mu_F(0)= \int_{\T} e(qmx)\d\mu_f(x) + \Oh(\epsilon),
\]
where the error term $\Oh(\epsilon)$ is independent of $m$. Averaging over $m$ and using \eqref{eqn_spectral_char_per_aper_1}, we obtain
\begin{align*}
\mu_f(\T)
= \mu_f\big(\big\{0,\tfrac{1}{q},\ldots,\tfrac{q-1}{q}\big\}\big)+ \Oh(\epsilon).
\end{align*}
In other words, up to an $\epsilon$-error, the measure $\mu_f$ is supported on the set $\big\{0,\frac{1}{q},\ldots,\frac{q-1}{q}\big\}$. Letting $\epsilon$ go to $0$ proves that $\mu_f$ is supported on the rationals.

To prove \ref{itm_spectral_char_per_aper_ii}, assume $f$ is locally totally ergodic. From the definition, it follows that for all $q\in\N$ we have
\[
\lim_{M\to\infty} \E_{m\in[M]}\hat\mu_f(qm)=0.
\]
In light of \eqref{eqn_spectral_char_per_aper_1}, this means that $\mu_f\big(\big\{0,\frac{1}{q},\ldots,\frac{q-1}{q}\big\}\big)=0$ for all $q\in\N$. But if this holds for all $q\in\N$ then the measure $\mu_f$ must be supported on the irrationals. 
\end{proof}

\subsection{A multiplicative analogue of van der Corput's inequality}
\label{sec_multiplicative_vdC}

Van der Corput's fundamental inequality (see~\cite[Lemma 3.1]{KN74}) states that for any $1$-bounded function $f\colon\Z\to\C$ and any $N,H\in\N$ with $1\leq H\leq N$ one has
\begin{equation}
\label{eqn_vdC1}
\Big|\E_{h\in [H]} \E_{n\in[N]} f(n) \Big|^2
\leq \Re\bigg(
\E_{h\in[H]} \Big(\tfrac{H-h}{H}\Big) \E_{n\in[N]} f(n)\overline{f(n+h)} \bigg)+\Oh\bigg(\frac{1}{H}+\frac{H}{N}\bigg).
\end{equation}
It is an instrumental inequality used for proving correlation estimates, convergence theorems, and uniform distribution results in number theory, harmonic analysis, and ergodic theory.
A small refinement of the classical formulation of van der Corput's inequality, and one that is often overlooked, asserts that for any $1$-bounded functions $f,g\colon\Z\to\C$ and any $N,H\in\N$ with $1\leq H\leq N$,
\begin{equation}
\label{eqn_vdC2}
\Big|\E_{h\in [H]} \E_{n\in[N]} f(n) g(n+h) \Big|^2
\leq \Re\bigg(
\E_{h\in[H]} \Big(\tfrac{H-h}{H}\Big) \E_{n\in[N]} f(n)\overline{f(n+h)} \bigg)+\Oh\bigg(\frac{1}{H}+\frac{H}{N}\bigg).
\end{equation}
We remark that the standard proof of \eqref{eqn_vdC1} (such as the one given in \cite[Lemma 3.1]{KN74}) with minimal adjustments proves \eqref{eqn_vdC2} as well.

One can also derive a logarithmically averaged version of van der Corput's inequality: For any 
$1$-bounded functions $f,g\colon\Z\to\C$ and any $N,H\in\N$ with $1\leq H\leq N$ we have
\begin{equation}
\label{eqn_vdC3}
\Big|\logE_{h\in [H]} \logE_{n\in[N]} f(n) g(n+h) \Big|^2
\leq \Re\bigg(
\logE_{h\in[H]} \logE_{n\in[N]} f(n)\overline{f(n+h)} \bigg)+\Oh\bigg(\frac{1}{\log(H)}+\frac{\log(H)}{\log(N)}\bigg).
\end{equation}
Since we do not make use of \eqref{eqn_vdC3} in this paper, we omit its proof. It merely serves as motivation for the following novel analogue of \eqref{eqn_vdC3} that uses multiplicative instead of additive differences and plays a crucial role in our argument.

\begin{theorem}
\label{thm_mvc}
Let $y\geq 1$, $P\subset \P\cap [y]$, and suppose $f_p,g\colon \Z\to\C$ are $1$-bounded functions for all $p\in P$. Then
\[
\Big|\logE_{p\in P} \logE_{n\in[N]} f_p(n)g(pn)\Big|^2
\,\leq\, \Re\bigg(
\logE_{p,q\in P} \logE_{n\in[N]} f_p(qn)\overline{f_q(pn)} \bigg)+\Oh\Bigg(\bigg(\sum_{p\in P}\frac{1}{p} \bigg)^{-1}+ \frac{\log y}{\log N}\Bigg).
\]
\end{theorem}

\begin{remark}
If $g(n)$ is a multiplicative function (i.e.~$g(nm)=g(n)g(m)$ whenever $\gcd(n,m)=1$) and $f_p(n)=f(n)$ for some $f\colon\Z\to\C$ and all $p\in P$, then \cref{thm_mvc} is closely related to the Daboussi-Delange-\Katai{}-Bourgain-Sarnak-Ziegler orthogonality criterion \cite{DD82, Katai86, BSZ13}.
\end{remark}

\begin{proof}[Proof of \cref{thm_mvc}]
Using \cref{lem_dilation_log_averages}, we can write
\begin{align*}
\Big|\logE_{p\in P} \logE_{n\in[N]} f_p(n)g(pn)\Big|^2
&= \Big|\logE_{p\in P} \logE_{n\in[N]}p \1_{p\mid n} f_p\big(\tfrac{n}{p}\big)g(n)  \Big|^2 + \Oh\bigg(\frac{\log y}{\log N}\bigg)
\\
&\leq \logE_{n\in[N]}\Big|\logE_{p\in P}  p \1_{p\mid n} f_p\big(\tfrac{n}{p}\big)\Big|^2 + \Oh\bigg(\frac{\log y}{\log N}\bigg)
\\
&= \logE_{p,q\in P} \logE_{n\in[N]} pq \1_{\operatorname{lcm}(p,q)\mid n} f_p\big(\tfrac{n}{p}\big)\overline{f_q\big(\tfrac{n}{q}\big)} + \Oh\bigg(\frac{\log y}{\log N}\bigg)
\\
&= 
\logE_{p,q\in P} \logE_{n\in[N]} pq \1_{pq\mid n} f_p\big(\tfrac{n}{p}\big)\overline{f_q\big(\tfrac{n}{q}\big)} + [1]+ \Oh\bigg(\frac{\log y}{\log N}\bigg),
\end{align*}
where
\begin{align*}
[1]
&= \logE_{p,q\in P} \1_{p=q}\Big( \logE_{n\in[N]} p^2 \1_{p\mid n} \big|f_p\big(\tfrac{n}{p}\big)\big|^2 - \logE_{n\in[N]} p^2 \1_{p^2\mid n} \big|f_p\big(\tfrac{n}{p}\big)\big|^2\Big)
\\
&\leq
\logE_{p,q\in P} \1_{p=q}
\Big( \logE_{n\in[N]} p^2 \1_{p\mid n} \Big)
\\
&\leq \bigg(\sum_{p\in P}\frac{1}{p} \bigg)^{-2} \sum_{p\in P}\frac{1}{p^2}\Big(\logE_{n\in[N]}~p^2\1_{p\mid n}\Big)
\\
&\leq \bigg(\sum_{p\in P}\frac{1}{p} \bigg)^{-1}.
\end{align*}
Using \cref{lem_dilation_log_averages} once more, we see that
\begin{align*}
\logE_{p,q\in P} \logE_{n\in[N]} pq \1_{pq\mid n} f_p\big(\tfrac{n}{p}\big)\overline{f_q\big(\tfrac{n}{q}\big)} 
=
\Re\Big(\logE_{p,q\in P} \logE_{n\in[N]} f_p(qn)\overline{f_q(pn)}\Big) + \Oh\bigg(\frac{\log y}{\log N}\bigg).
\end{align*}
Combining all the estimates above completes the proof.
\end{proof}

\subsection{Proof of \cref{thm_rnd_component}}
\label{sec_proof_of_thm_rnd_component}

We now have all the necessary ingredients to provide a proof of \cref{thm_rnd_component}. For the convenience of the reader, let us restate the theorem here.

\begin{named}{\cref{thm_rnd_component}}{}{}
Let $N_1<N_2<\ldots\in\N$ and $M_1<M_2<\ldots\in\N$. Suppose $\A\subset\ell^\infty(\Z)$ is a separable, affinely invariant $C^*$-subalgebra that admits logarithmic averages along $\bN=([N_s])_{s\in\N}$ and iterated affine correlations along $\bPW=(\P_W\cap [M_t])_{t\in\N}$ for all $W\in\N$.
If $f\in\A$ is locally totally ergodic with respect to logarithmic averages along $\bN$, then for all $k,W,a\in\N$ with $k\geq 2$, all $g\in\A$, and all non-constant polynomials $Q$ with integer coefficients we have
\[
\logE_{p\in\bPW^{*k}}\logE_{n\in\bN} f\big(an+Q(p)\big)g(pn)=0.
\]
\end{named}

\begin{proof}
It suffices to prove the case $k=2$.
Using the Cauchy-Schwarz inequality and \cref{thm_mvc}, we obtain
\begin{align*}
\Big|\logE_{p_2\in \bPW}\logE_{p_1\in \bPW} & \logE_{n\in \bN} f(an+Q(p_2p_1))g(p_2p_1n)\Big|^2
\\
&=\lim_{t\to\infty}\Big|\logE_{p_2\in \P_W\cap[M_t]} \logE_{p_1\in \bPW} \logE_{n\in \bN} f(an+Q(p_2p_1))g(p_2p_1n)\Big|^2
\\
&\leq \lim_{t\to\infty} \logE_{p_1\in \bPW} \Big|\logE_{p_2\in \P_W\cap[M_t]} \logE_{n\in \bN} f(an+Q(p_2p_1))g(p_2p_1n)\Big|^2
\\
&\leq \lim_{t\to\infty} \logE_{p_1\in \bPW} \Big(\Re\Big(\logE_{p_2,p_3\in \P_W\cap[M_t]} \logE_{n\in \bN} f(ap_3n+Q(p_2p_1))\overline{f(ap_2n+Q(p_3p_1))}\Big)\Big)
\\
&= \lim_{t\to\infty}\logE_{p_2,p_3\in \P_W\cap[M_t]}\Re\Big(  \logE_{p_1\in \bPW} \logE_{n\in \bN} f(ap_3n+Q(p_2p_1))\overline{f(ap_2n+Q(p_3p_1))}\Big).
\end{align*}
We claim that whenever $p_2\neq p_3$ then
\begin{equation}
\label{eqn_locally_aperiodic_have_no_contribution_1}
\logE_{p_1\in \bPW} \logE_{n\in \bN} f(ap_3n+Q(p_2p_1))\overline{f(ap_2n+Q(p_3p_1))}=0.
\end{equation}
Once this claim has been verified, it follows that only the diagonal terms $p_2=p_3$ contribute, which implies that
\begin{align*}
\Big|\logE_{p_2\in \bPW}\logE_{p_1\in \bPW} & \logE_{n\in \bN} f(an+Q(p_2p_1))g(p_2p_1n)\Big|^2
\leq \|f\|_\infty^2  \Big(\lim_{t\to\infty}\logE_{p_2,p_3\in \P_W\cap[M_t]} \1_{p_2=p_3}\Big)
=0.
\end{align*}

It remains to verify \eqref{eqn_locally_aperiodic_have_no_contribution_1} under the assumption $p_2\neq p_3$. By symmetry, we can assume without loss of generality that $p_3>p_2$.
Let $q=ap_2p_3$. 
First, we split the right hand side of \eqref{eqn_locally_aperiodic_have_no_contribution_1} into residue classes mod $q$.
Observe that for all $m\in\Z$,
\[
Q(p_2(qm+r))\equiv Q(p_2r)\bmod q\qquad\text{and}\qquad
Q(p_3(qm+r))\equiv Q(p_3r)\bmod q.
\]
Define
\[
Q_{2,r}(n)=\frac{Q(p_2n)-Q(p_2r)}{ap_3}\qquad\text{and}\qquad
Q_{3,r}(n)=\frac{Q(p_3n)-Q(p_3r)}{ap_2}.
\]
Therefore
\begin{align*}
&\logE_{p_1\in \bPW}  \logE_{n\in \bN} f(ap_3n+Q(p_2p_1))\overline{f(ap_2n+Q(p_3p_1))}
\\
&~=\sum_{r=0}^{q-1} \logE_{p_1\in \bPW} \logE_{n\in \bN} \1_{q\Z+r}(p_1)\, f(ap_3n+Q(p_2p_1))\overline{f(ap_2n+Q(p_3p_1))}
\\
&~=\sum_{r=0}^{q-1} \logE_{p_1\in \bPW} \logE_{n\in \bN} \1_{q\Z+r}(p_1)\, f\big(ap_3\big(n+Q_{2,r}(p_1)\big)+Q(p_2r)\big)
\\
&\hspace{19em}\overline{f\big(ap_2\big(n+Q_{3,r}(p_1)\big)+Q(p_3r)\big)}.
\end{align*}
To simplify notation, let us define, for $r\in\{0,\ldots,q-1\}$, the functions 
\[
f_{r,1}(n)=f(ap_3n+Q(p_2r))\qquad\text{and}\qquad f_{r,2}(n)=f(ap_2n+Q(p_3r)).
\]
Then we can rewrite the above as
\begin{align*}
\logE_{p_1\in \bPW} & \logE_{n\in \bN} f(ap_3n+Q(p_2p_1))\overline{f(ap_2n+Q(p_3p_1))}
\\
&~=\sum_{r=0}^{q-1} \logE_{p_1\in \bPW} \logE_{n\in \bN} \1_{q\Z+r}(p_1)\, f_{r,1}(n+Q_{2,r}(p_1))\overline{f_{r,2}(n+Q_{3,r}(p_1))}
\\
&~=\sum_{r=0}^{q-1} \logE_{p_1\in \bPW} \logE_{n\in \bN} \1_{q\Z+r}(p_1)\, f_{r,1}(n)\overline{f_{r,2}(n+Q_{3,r}(p_1)-Q_{2,r}(p_1))}.
\end{align*}
After exchanging the order of summation in the variables $n$ and $p_1$, using the Cauchy-Schwarz inequality, and finally expanding the square, we can eliminate the function $f_{r,1}$ and obtain a single correlation expression in the function $f_{r,2}$:
\begin{align*}
\Big|\logE_{p_1\in \bPW} & \logE_{n\in \bN} f(ap_3n+Q(p_2p_1))\overline{f(ap_2n+Q(p_3p_1))}\Big|
\\
&~=\bigg|\sum_{r=0}^{q-1} \Big(\lim_{t\to\infty}\logE_{p_1\in \P_W\cap [M_t]} \logE_{n\in \bN} \1_{q\Z+r}(p_1)\, f_{r,1}(n)\overline{f_{r,2}(n+Q_{3,r}(p_1)-Q_{2,r}(p_1))}\Big)\bigg|
\\
&~\leq \|f\|_\infty  \sum_{r=0}^{q-1}\lim_{t\to\infty} \bigg( \logE_{n\in \bN}  \Big|\logE_{p_1\in \P_W\cap [M_t]}\1_{q\Z+r}(p_1)\, f_{r,2}(n+Q_{3,r}(p_1)-Q_{2,r}(p_1))\Big|^2\bigg)^{\frac{1}{2}}.
\end{align*}
Let $\mu_{f_{r,2}}$ be the spectral measure associated to $f_{r,2}$, as defined in \eqref{eqn_spectral_measure}. 
In light of \cref{lem_spectral_measure_isometry_property}, it follows that
\begin{align*}
\lim_{t\to\infty} \bigg( \logE_{n\in \bN} &  \Big|\logE_{p_1\in \P_W\cap [M_t]}\1_{q\Z+r}(p_1)\, f_{r,2}(n+Q_{3,r}(p_1)-Q_{2,r}(p_1))\Big|^2\bigg)^{\frac{1}{2}}
\\
&= \lim_{t\to\infty} \bigg(\int_\T \Big|\logE_{p_1\in \P_W\cap [M_t]}\1_{q\Z+r}(p_1)\, e\big(\big(Q_{3,r}(p_1)-Q_{2,r}(p_1)\big)x\big)\Big|^2\d\mu_{f_{r,2}}(x)\bigg)^{\frac{1}{2}}.
\end{align*}
Since $Q$ is non-constant and $p_2$ and $p_3$ are distinct primes, the polynomial $x\mapsto Q_{3,r}(x)-Q_{2,r}(x)$ is non-constant with rational coefficients.
According to \cref{cor_log_vinogradov}, if $x$ is irrational then
\[
\lim_{t\to\infty} \logE_{p_1\in \P_W\cap [M_t]}\1_{q\Z+r}(p_1)\, e\big(\big(Q_{3,r}(p_1)-Q_{2,r}(p_1)\big)x\big)=0.
\]
By assumption, $f$ is locally totally ergodic. Due to \cref{lem_basic_properties_locally_periodic_locally_aperiodic}, this means that $f_{r,2}$ is also locally totally ergodic. In view of the spectral characterization given by \cref{thm_spectral_char_per_aper}, the spectral measure $\mu_{f_{r,2}}$ of $f_{r,2}$ is supported on irrational numbers. It thus follows by the Lebesgue's dominated convergence theorem that
\[
\lim_{t\to\infty} \bigg(\int_\T \Big|\logE_{p_1\in \P_W\cap [M_t]}\1_{q\Z+r}(p_1)\, e\big(\big(Q_{3,r}(p_1)-Q_{2,r}(p_1)\big)x\big)\Big|^2\d\mu_{f_{r,2}}(x)\bigg)^{\frac{1}{2}}=0.
\]
This concludes the proof.
\end{proof}

\section{Controlling the ``structured'' component}
\label{sec_proof_of_thm_str_component}

In this section, we give a proof of \cref{thm_str_component}.
Let us restate the theorem:

\begin{named}{\cref{thm_str_component}}{}
Let $N_1<N_2<\ldots\in\N$ and $M_1<M_2<\ldots\in\N$. Suppose $\A\subset\ell^\infty(\Z)$ is a separable, affinely invariant $C^*$-subalgebra that admits logarithmic averages along $\bN=([N_s])_{s\in\N}$ and iterated affine correlations along $\bPW=(\P_W\cap [M_t])_{t\in\N}$ for all $W\in\N$.
Suppose $a_1,a_2,\ldots\in\N$ and $f\in\A$ with $f(n)\geq 0$ for all $n\in\N$.
Then for any $\epsilon>0$ and any $W\in\N$ there are $k\geq 2$ and an infinite set $I\subset\N$, such that for all $i,j\in I$ with $i<j$ we have
\begin{equation}
\label{eqn_improved_DE}
\logE_{p\in\bPW^{*k}} \logE_{n\in \bN} f(a_in)f(a_jpn) \geq  \Big(\limsup_{j\to\infty}\logE_{n\in\bN} f(a_jn)\Big)^2-\epsilon.    
\end{equation}
\end{named}

\begin{proof}
Let $\epsilon>0$ be given, and define $\delta=\limsup_{j\to\infty}\E^{\log}_{n\in\bN} f(a_jn)$.
By replacing $(a_j)_{j\in\N}$ with a subsequence of itself, we can assume without loss of generality that
\begin{equation}
\label{eqn_improved_DE_2}
\logE_{n\in\bN} f(a_jn)\geq \delta-\frac{\epsilon}{3},\qquad\forall j\in\N.
\end{equation}
Let $K\in\N$ be any natural number that satisfies $K>3\epsilon^{-1}$.
We now define a $K$ coloring of the set $\{(i,j)\in\N\times\N: i<j\}$, and 
we represent this coloring as a map $\chi\colon \{(i,j)\in\N\times\N: i<j\} \to \{1,\ldots,K\}$.
Given $i,j\in\N$ with $i<j$, define
\[
\chi(i,j)=
\begin{cases}
k,&\text{if $k$ is the smallest number in $\{2,\ldots,K\}$ for which \eqref{eqn_improved_DE} holds;}
\\
1,&\text{if \eqref{eqn_improved_DE} doesn't hold for any $k\in\{2,\ldots,K\}$}.
\end{cases}
\]
By Ramsey's theorem \cite[Theorem~A]{Ramsey30}, there exists an infinite set $I\subset \N$ such that $\chi(i,j)$ has the same color for all $i,j\in I$ with $i<j$.
If this color belongs to the set $\{2,\ldots,K\}$ then we are done.
So it remains to show that $1$ is not an admissible color.
By way of contradiction, suppose $\chi(i,j)=1$ for all $i,j\in I$ with $i<j$.
This implies that all $i,j\in I$ with $i<j$,
\[
\max_{k=2,\ldots,K}\Big(\logE_{p\in\bPW^{*k}} \logE_{n\in \bN} f(a_in)f(a_jpn)\Big) < \delta^2-\epsilon.    
\]
Let $i_2<\ldots<i_K\in I$ be arbitrary. 
Using \eqref{eqn_improved_DE_2}, \cref{cor_log_TK_2}, and the Cauchy-Schwarz inequality, we have
\begin{align*}
\Big(\delta-\frac{\epsilon}{3}\Big)^2 &\leq \Big( \E_{k\in\{2,\ldots,K\}} \logE_{n\in\bN} f(a_{i_k}n) \Big)^2
\\
&= \Big( \E_{k\in\{2,\ldots,K\}} \logE_{p\in\bPW^{*k}} \logE_{n\in\bN} f(a_{i_k}pn) \Big)^2
\\
&= \Big( \logE_{p_K\in\bPW} \cdots \logE_{p_1\in\bPW} \logE_{n\in\bN} \E_{k\in\{2,\ldots,K\}}  f(a_{i_k}p_k\cdots p_1n) \Big)^2
\\
&\leq \logE_{p_K\in\bPW} \cdots \logE_{p_1\in\bPW} \logE_{n\in\bN} \Big( \E_{k\in\{2,\ldots,K\}} f(a_{i_k}p_k\cdots p_1n) \Big)^2
\\
&= \E_{k,\ell\in\{2,\ldots,K\}} \logE_{p_K\in\bPW} \cdots \logE_{p_1\in\bPW} \logE_{n\in\bN} f(a_{i_\ell} p_\ell\cdots p_1n) f(a_{i_k}p_k\cdots p_1n)
\\
&\leq  \E_{k,\ell\in\{2,\ldots,K\}} 2\1_{\ell<k} \logE_{p_K\in\bPW} \cdots \logE_{p_1\in\bPW} \logE_{n\in\bN} f(a_{i_\ell} p_\ell\cdots p_1n) f(a_{i_k}p_k\cdots p_1n) + \frac{1}{K}
\\
&=  \E_{k,\ell\in\{2,\ldots,K\}} 2\1_{\ell<k} \logE_{p_K\in\bPW} \cdots \logE_{p_1\in\bPW} \logE_{n\in\bN} f(a_{i_\ell} n) f(a_{i_k}p_k\cdots p_{k-\ell+1}n) + \frac{1}{K-1}
\\
&=  \E_{k,\ell\in\{2,\ldots,K\}} 2\1_{\ell<k} \logE_{p\in\bPW^{*\ell}} \logE_{n\in\bN} f(a_{i_\ell} n) f(a_{i_k}pn) + \frac{1}{K-1}
\\
&< \delta^2-\epsilon+\frac{1}{K-1}.
\end{align*}
This contradicts the assumption that $K>3\epsilon^{-1}$, completing the proof.
\end{proof}

\section{Open questions}

To conclude, we formulate some open questions and conjectures connected to our results and the broader topic.

We begin by asking for a natural extension of \cref{thm_2}.

\begin{question}
Is there a multiplicatively invariant density on $\N$ (similar to $\AMdens(.)$) such that any set with positive density with respect to this density notion contains $\{x,x+y,xy\}$ for some $x,y\in\N$. 
\end{question}

The next question seeks a quantitative version of (a special case of) Moreira's theorem \cite{Moreira17}.
Given $r\in\N$, let $M(r)$ denote the smallest positive integer such that for all $N\geq M(r)$ and all $r$-colorings of $[N]$ one can find $x,y\in[N]$ with $x>y>2$ such that $\{x+y,xy\}\subset[N]$ is monochromatic. 
It follows from the compactness principle (see~\cite[Section~1.5]{GRS90}) applied to Moreira's theorem that $M(r)$ is well defined for every $r\in\N$.

\begin{question}
What lower and upper bounds on $M(r)$ can be provided?
\end{question}

\begin{remark}
After a preprint of this paper appeared on arXiv, it was shown by Green and Sawhney \cite{GreenSawhney25arXiv} that
$r\ll \log(M(r)) \ll e^{r^{50}}$.
\end{remark}

Finally, it is natural to inquire about sums and products in the set $\P-1$.

\begin{conjecture}
For every $a\in\N$ there exist $x,y\in\N$ with $x>y>a$ such that 
$\{x+y,xy\}\subset \P-1$.
\end{conjecture}

\bibliographystyle{aomalphanomr}
\bibliography{mynewlibrary.bib,ProjectSpecificReferences.bib}

\providecommand{\bysame}{\leavevmode\hbox to3em{\hrulefill}\thinspace}
\providecommand{\noopsort}[1]{}
\providecommand{\zbl}[1]{\href{http://www.zentralblatt-math.org/zmath/en/search/?q=an:#1}{Zbl~#1}}
\providecommand{\jfm}[1]{\href{http://www.emis.de/cgi-bin/JFM-item?#1}{JFM~#1}}
\providecommand{\arxiv}[1]{\href{http://www.arxiv.org/abs/#1}{arXiv~#1}}
\providecommand{\doi}[1]{\url{https://doi.org/#1}}
\providecommand{\href}[2]{#2}
\begin{thebibliography}{KMRR25}

\bibitem[AKS99]{AKS99}
\bgroup\scshape{}R.~Ahlswede\egroup{},
  \bgroup\scshape{}L.~Khachatrian\egroup{}, and
  \bgroup\scshape{}A.~Sárközy\egroup{}, On the quotient sequence of sequences
  of integers,  \emph{Acta Arithmetica} \textbf{91} (1999), 117--132.
  \doi{10.4064/aa-91-2-117-132}.

\bibitem[Alw23]{Alweiss23}
\bgroup\scshape{}R.~Alweiss\egroup{}, Monochromatic sums and products over
  $\mathbb{Q}$, 2023. Available at \url{https://arxiv.org/abs/2307.08901}.

\bibitem[Alw24]{Alweiss24}
\bgroup\scshape{}R.~Alweiss\egroup{}, {M}onochromatic {S}ums and {P}roducts of
  {P}olynomials,  \emph{Discrete Analysis} (2024). \doi{10.19086/da.117575}.

\bibitem[BBHS06]{BBHS06}
\bgroup\scshape{}M.~Beiglb{\"o}ck\egroup{},
  \bgroup\scshape{}V.~Bergelson\egroup{}, \bgroup\scshape{}N.~Hindman\egroup{},
  and \bgroup\scshape{}D.~Strauss\egroup{}, Multiplicative structures in
  additively large sets,  \emph{J. Combin. Theory Ser. A} \textbf{113} no.~7
  (2006), 1219--1242. \doi{10.1016/j.jcta.2005.11.003}.

\bibitem[Ber96]{Bergelson96}
\bgroup\scshape{}V.~Bergelson\egroup{}, Ergodic {R}amsey theory -- an update,
  in \emph{Ergodic theory of $\mathbb{Z}^d$ actions (Warwick, 1993--1994)},
  \emph{London Math. Soc. Lecture Note Ser.} \textbf{228}, Cambridge Univ.
  Press, Cambridge, 1996, pp.~1--61. \doi{10.1017/CBO9780511662812.002}.

\bibitem[BM17]{BM17}
\bgroup\scshape{}V.~Bergelson\egroup{} and
  \bgroup\scshape{}J.~Moreira\egroup{}, Ergodic theorem involving additive and
  multiplicative groups of a field and {$\{x+y,xy\}$} patterns,  \emph{Ergodic
  Theory Dynam. Systems} \textbf{37} no.~3 (2017), 673--692.
  \doi{10.1017/etds.2015.68}.

\bibitem[BM18]{BM18}
\bgroup\scshape{}V.~Bergelson\egroup{} and
  \bgroup\scshape{}J.~Moreira\egroup{}, Measure preserving actions of affine
  semigroups and {$\{x+y,xy\}$} patterns,  \emph{Ergodic Theory Dynam. Systems}
  \textbf{38} no.~2 (2018), 473--498. \doi{10.1017/etds.2016.39}.

\bibitem[BR09]{BR09}
\bgroup\scshape{}V.~Bergelson\egroup{} and
  \bgroup\scshape{}B.~Rothschild\egroup{}, A selection of open problems,
  \emph{Topology Appl.} \textbf{156} no.~16 (2009), 2674--2681.
  \doi{10.1016/j.topol.2009.04.020}.

\bibitem[Bes35]{Besicovitch35}
\bgroup\scshape{}A.~S. Besicovitch\egroup{}, On the density of certain
  sequences of integers,  \emph{Math. Ann.} \textbf{110} no.~1 (1935),
  336--341. \doi{10.1007/BF01448032}.

\bibitem[BSZ13]{BSZ13}
\bgroup\scshape{}J.~Bourgain\egroup{}, \bgroup\scshape{}P.~Sarnak\egroup{}, and
  \bgroup\scshape{}T.~Ziegler\egroup{}, Disjointness of moebius from horocycle
  flows,  in \emph{From {F}ourier analysis and number theory to radon
  transforms and geometry}, \emph{Dev. Math.} \textbf{28}, Springer, New York,
  2013, pp.~67--83. \doi{10.1007/978-1-4614-4075-8_5}.

\bibitem[Bow25]{Bowen25}
\bgroup\scshape{}M.~Bowen\egroup{}, Monochromatic products and sums in
  2-colorings of {$\Bbb{N}$},  \emph{Adv. Math.} \textbf{462} (2025), Paper No.
  110095, 17. \doi{10.1016/j.aim.2024.110095}.

\bibitem[BS24]{BS24}
\bgroup\scshape{}M.~Bowen\egroup{} and \bgroup\scshape{}M.~Sabok\egroup{},
  Monochromatic products and sums in the rationals,  \emph{Forum Math. Pi}
  \textbf{12} (2024), Paper No. e17, 12. \doi{10.1017/fmp.2024.19}.

\bibitem[DD82]{DD82}
\bgroup\scshape{}H.~Daboussi\egroup{} and \bgroup\scshape{}H.~Delange\egroup{},
  On multiplicative arithmetical functions whose modulus does not exceed one,
  \emph{J. London Math. Soc. (2)} \textbf{26} no.~2 (1982), 245--264.
  \doi{10.1112/jlms/s2-26.2.245}.

\bibitem[DE36]{DE36}
\bgroup\scshape{}H.~Davenport\egroup{} and
  \bgroup\scshape{}P.~Erd\H{o}s\egroup{}, On sequences of positive integers,
  \emph{Acta Arithmetica} \textbf{2} no.~1 (1936), 147--151 (eng). Available at
  \url{http://eudml.org/doc/205438}.

\bibitem[DE51]{DE51}
\bgroup\scshape{}H.~Davenport\egroup{} and
  \bgroup\scshape{}P.~Erd\H{o}s\egroup{}, On sequences of positive integers,
  \emph{J. Indian Math. Soc. (N.S.)} \textbf{15} (1951), 19--24.

\bibitem[Ell79]{Elliott79}
\bgroup\scshape{}P.~D. T.~A. Elliott\egroup{}, \emph{Probabilistic number
  theory. {I}}, \emph{Grundlehren der Mathematischen Wissenschaften
  [Fundamental Principles of Mathematical Science]} \textbf{239},
  Springer-Verlag, New York-Berlin, 1979, Mean-value theorems.
  \doi{10.1007/978-1-4612-9989-9}.

\bibitem[Erd35]{Erdos35}
\bgroup\scshape{}P.~Erd\H{o}s\egroup{}, Note on {S}equences of {I}ntegers {N}o
  {O}ne of {W}hich is {D}ivisible {B}y {A}ny {O}ther,  \emph{J. London Math.
  Soc.} \textbf{10} no.~2 (1935), 126--128. \doi{10.1112/jlms/s1-10.1.126}.

\bibitem[Far24]{Farhangi24}
\bgroup\scshape{}S.~Farhangi\egroup{}, A generalization of van der {C}orput's
  difference theorem with applications to recurrence and multiple ergodic
  averages,  \emph{Dynamical Systems} \textbf{39} no.~1 (2024), 5--30.
  \doi{10.1080/14689367.2023.2230160}.

\bibitem[Fra24]{Frantzikinakis24}
\bgroup\scshape{}N.~Frantzikinakis\egroup{}, Partition regularity of
  homogeneous quadratics: Current trends and challenges, 2024. Available at
  \url{https://arxiv.org/abs/2411.17523}.

\bibitem[GRS90]{GRS90}
\bgroup\scshape{}R.~L. Graham\egroup{}, \bgroup\scshape{}B.~L.
  Rothschild\egroup{}, and \bgroup\scshape{}J.~H. Spencer\egroup{},
  \emph{Ramsey theory}, second ed., \emph{Wiley-Interscience Series in Discrete
  Mathematics and Optimization}, John Wiley \& Sons, Inc., New York, 1990, A
  Wiley-Interscience Publication.

\bibitem[GS16]{GS16}
\bgroup\scshape{}B.~Green\egroup{} and \bgroup\scshape{}T.~Sanders\egroup{},
  Monochromatic sums and products,  \emph{Discrete Anal.} (2016), Paper No. 5,
  43. \doi{10.19086/da.613}.

\bibitem[GS25]{GreenSawhney25arXiv}
\bgroup\scshape{}B.~Green\egroup{} and \bgroup\scshape{}M.~Sawhney\egroup{},
  Bounds for monochromatic solutions to $\{x+y,xy\}$, 2025. Available at
  \url{https://arxiv.org/abs/2511.09365}.

\bibitem[HS12]{HS12a}
\bgroup\scshape{}N.~Hindman\egroup{} and \bgroup\scshape{}D.~Strauss\egroup{},
  \emph{{A}lgebra in the {S}tone-\v{C}ech {C}ompactification -- {T}heory and
  {A}pplications}, \emph{de Gruyter Textbook}, Walter de Gruyter \& Co.,
  Berlin, 2012, Second revised and extended edition.
  \doi{10.1515/9783110258356}.

\bibitem[Hin79]{Hindman79a}
\bgroup\scshape{}N.~Hindman\egroup{}, Partitions and sums and products of
  integers,  \emph{Trans. Amer. Math. Soc.} \textbf{247} (1979), 227--245.
  \doi{10.2307/1998782}.

\bibitem[HLS03]{HLS03}
\bgroup\scshape{}N.~Hindman\egroup{}, \bgroup\scshape{}I.~Leader\egroup{}, and
  \bgroup\scshape{}D.~Strauss\egroup{}, Open problems in partition regularity,
  \textbf{12}, 2003, Special issue on Ramsey theory, pp.~571--583.
  \doi{10.1017/S0963548303005716}.

\bibitem[K{\'a}t86]{Katai86}
\bgroup\scshape{}I.~K{\'a}tai\egroup{}, A remark on a theorem of {H}.
  {D}aboussi,  \emph{Acta Math. Hungar.} \textbf{47} no.~1-2 (1986), 223--225.
  \doi{10.1007/BF01949145}.

\bibitem[Kou26]{Kousek26}
\bgroup\scshape{}I.~Kousek\egroup{}, Revisiting sums and products in countable
  and finite fields,  \emph{Ergodic Theory and Dynamical Systems} \textbf{46}
  no.~2 (2026), 543--574. \doi{10.1017/etds.2025.10244}.

\bibitem[KMRR25]{KMRR25}
\bgroup\scshape{}B.~Kra\egroup{}, \bgroup\scshape{}J.~Moreira\egroup{},
  \bgroup\scshape{}F.~K. Richter\egroup{}, and
  \bgroup\scshape{}D.~Robertson\egroup{}, Problems on infinite sumset
  configurations in the integers and beyond,  \emph{Bulletin of the American
  Mathematical Society} \textbf{62} (2025), 537--574. \doi{10.1090/bull/1868}.

\bibitem[KN74]{KN74}
\bgroup\scshape{}L.~Kuipers\egroup{} and
  \bgroup\scshape{}H.~Niederreiter\egroup{}, \emph{Uniform distribution of
  sequences}, Wiley-Interscience [John Wiley \& Sons], New York-London-Sydney,
  1974, Pure and Applied Mathematics.

\bibitem[Mor17]{Moreira17}
\bgroup\scshape{}J.~Moreira\egroup{}, Monochromatic sums and products in {$\Bbb
  N$},  \emph{Ann. of Math. (2)} \textbf{185} no.~3 (2017), 1069--1090.
  \doi{10.4007/annals.2017.185.3.10}.

\bibitem[MRR19]{MRR19}
\bgroup\scshape{}J.~Moreira\egroup{}, \bgroup\scshape{}F.~K. Richter\egroup{},
  and \bgroup\scshape{}D.~Robertson\egroup{}, A proof of a sumset conjecture of
  {E}rd{\H{o}}s,  \emph{Ann. of Math. (2)} \textbf{189} no.~2 (2019), 605--652.
  \doi{10.4007/annals.2019.189.2.4}.

\bibitem[Pet83]{Petersen83}
\bgroup\scshape{}K.~Petersen\egroup{}, \emph{Ergodic theory}, \emph{Cambridge
  Studies in Advanced Mathematics} \textbf{2}, Cambridge University Press,
  Cambridge, 1983. \doi{10.1017/CBO9780511608728}.

\bibitem[Ram30]{Ramsey30}
\bgroup\scshape{}F.~P. Ramsey\egroup{}, On a {Problem} of {Formal} {Logic},
  \emph{Proceedings of the London Mathematical Society} \textbf{s2-30} no.~1
  (1930), 264--286. \doi{10.1112/plms/s2-30.1.264}.

\bibitem[Rhi73]{Rhin73}
\bgroup\scshape{}G.~Rhin\egroup{}, Sur la r\'{e}partition modulo {$1$} des
  suites {$f(p)$},  \emph{Acta Arith.} \textbf{23} (1973), 217--248.
  \doi{10.4064/aa-23-3-217-248}.

\bibitem[Sah18]{Sahasrabudhe18}
\bgroup\scshape{}J.~Sahasrabudhe\egroup{}, Exponential patterns in arithmetic
  {R}amsey theory,  \emph{Acta Arith.} \textbf{182} no.~1 (2018), 13--42.
  \doi{10.4064/aa8603-9-2017}.

\bibitem[Shk10]{Shkredov10}
\bgroup\scshape{}I.~D. Shkredov\egroup{}, On monochromatic solutions of some
  nonlinear equations in {$\Bbb Z/p\Bbb Z$},  \emph{Mat. Zametki} \textbf{88}
  no.~4 (2010), 625--634. \doi{10.1134/S0001434610090336}.

\bibitem[Tao16]{Tao16}
\bgroup\scshape{}T.~Tao\egroup{}, The logarithmically averaged {C}howla and
  {E}lliott conjectures for two-point correlations,  \emph{Forum Math. Pi}
  \textbf{4} (2016), e8, 36. \doi{10.1017/fmp.2016.6}.

\bibitem[Tao07]{Tao07b}
\bgroup\scshape{}T.~Tao\egroup{}, Structure and randomness in combinatorics,
  in \emph{48th Annual IEEE Symposium on Foundations of Computer Science
  (FOCS'07)}, 2007, pp.~3--15. \doi{10.1109/FOCS.2007.17}.

\bibitem[Vin57]{Vinogradov57}
\bgroup\scshape{}I.~M. Vinogradov\egroup{}, Trigonometric sums involving values
  of a polynomial,  \emph{Izv. Akad. Nauk SSSR Ser. Mat.} \textbf{21} (1957),
  145--170.

\bibitem[Vin58]{Vinogradov58}
\bgroup\scshape{}I.~M. Vinogradov\egroup{}, A special case of estimation of
  trigonometric sums involving prime numbers,  \emph{Izv. Akad. Nauk SSSR Ser.
  Mat.} \textbf{22} (1958), 3--14.

\end{thebibliography}


\bigskip
\footnotesize
\noindent
Florian K.\ Richter\\
\textsc{{\'E}cole Polytechnique F{\'e}d{\'e}rale de Lausanne (EPFL)}\\
\href{mailto:f.richter@epfl.ch}
{\texttt{f.richter@epfl.ch}}

\end{document}